\documentclass{article} 
\usepackage{amsthm, amsmath, amssymb}
\usepackage[colorlinks = true, linkcolor = black, citecolor = black, final]{hyperref}
\usepackage{graphicx, multicol}
\usepackage{marvosym, wasysym}
\usepackage{fancyhdr}
\usepackage{stackengine}
\usepackage{algorithm}
\usepackage{bm}
\usepackage[noend]{algpseudocode}
\usepackage{bbm}

\makeatletter
\renewcommand{\ALG@beginalgorithmic}{\scriptsize}
\makeatother

\usepackage[margin=1.2in]{geometry}

\newtheorem{theorem}{Theorem}

\newtheorem{result}{Result}

\newtheorem{assumption}{Assumption}

\newtheorem{lemma}[theorem]{Lemma}

\usepackage{multirow}

\def\delequal{\mathrel{\ensurestackMath{\stackon[1pt]{=}{\scriptscriptstyle\Delta}}}}
\def\distequal{\mathrel{\ensurestackMath{\stackon[1pt]{=}{\scriptstyle d}}}}
\newcommand{\norm}[1]{\left\lVert#1\right\rVert}
\usepackage{mathtools}
\DeclarePairedDelimiter{\ceil}{\lceil}{\rceil}

\algrenewcommand\algorithmicrequire{\textbf{Require:}}
\algrenewcommand\algorithmicensure{\textbf{Postcondition:}}
\usepackage{subfigure}
\usepackage{float}

\title{Efficient Rare-Event Simulation for Multiple Jump Events in Regularly Varying L\'evy Processes with Infinite Activities}

\author{
Xingyu Wang, Chang-Han Rhee\\
Department of Industrial Engineering and Management Sciences \\
Northwestern University, Evanston, IL, 60208  \\
\texttt{xingyuwang2017@u.northwestern.edu, chang-han.rhee@northwestern.edu} \\
}

\begin{document}

\maketitle
\begin{abstract}
In this paper we address the problem of rare-event simulation for heavy-tailed L\'evy processes with infinite activities. We propose a strongly efficient importance sampling algorithm that builds upon the sample path large deviations for heavy-tailed L\'evy processes, stick-breaking approximation of extrema of L\'evy processes, and the randomized debiasing Monte Carlo scheme. The proposed importance sampling algorithm can be applied to a broad class of L\'evy processes and exhibits significant improvements in efficiency when compared to crude Monte-Carlo method in our numerical experiments.
\end{abstract}
\section{INTRODUCTION}
\label{sec:intro}

In this paper, we propose a strongly efficient rare-event simulation algorithm for general L\'evy processes with heavy-tailed jump distributions characterized by regular variation. Specifically, our goal is to estimate probabilities of the form $\mathbb{P}(A_n)$ for large $n$, where $A_n = \{\bar{X}_n \in A\}$, $A$ is a subset of the Skorokhod path space, and $\bar{X}_n$ is a scaled L\'evy process $X$ with heavy-tailed jump distributions characterized by regular variation.
Such problems arise in many different applications such as ruin and risk theory \cite{asmussen2010ruin}, option pricing \cite{tankov2003financial}, and queuing networks \cite{dkebicki2015queues}.

Two major challenges arise when designing an efficient rare-event simulation algorithm for general L\'evy processes with heavy-tailed jumps. First, the nature of the rare events renders the crude Monte-Carlo method extremely inefficient when $n$ is large: assume the goal is to estimate $\mathbb{P}(A_n)$ with a given level of confidence on its relative error, then the number of samples required would approach $\infty$ as $n \rightarrow \infty$ and $\mathbb{P}(A_n)$ tends to $0$. In the light-tailed case, one typical solution is to perform an exponential change of measure and analyze a properly tilted process that induces a much higher probability of occurrence for the desired event. The theory of large deviations can be used to determine the right amount of exponential tilting.
For instance, when viewing risk processes from the perspective of large deviations, the asymptotic distribution of sample paths leading to ruination coincides with the distribution of the exponentially biased risk process parametrized by the solution of the Lundberg equation \cite{asmussen2010ruin}; and under the guidance of large deviation principles, importance sampling algorithms have been proposed to asymptotically optimally simulate rare events in a dynamic fashion \cite{dupuis2004importance}, or simulate rare events in queuing networks with established bound on required computational efforts \cite{boxma2019linear}.
Similarly, for L\'evy processes with heavy-tailed jumps, one would expect that the design of an efficient rare-event simulation algorithm entails the knowledge of large deviation results for the associated processes, since the large deviation principles not only characterize the decaying rate of $\mathbb{P}(A_n)$, but also describe the most likely scenario for the event to occur. Indeed, as revealed in \cite{rhee2016sample}, by solving an optimization problem concerning the minimal number of jumps $l^*$ required for a step function to trigger the target event, we see that, asymptotically, the sample paths leading to occurrence of rare events $A_n$ are those with $l^*$ large jumps. By exploiting this result to design a proper importance sampling distribution, \cite{chen2019efficient} proposes a strongly efficient rare-event simulation algorithm for compound Poisson processes and random walks with regularly varying jumps. 
The current paper extends this framework, and presents an importance sampling algorithm for rare-event simulation of general L\'evy processes with regularly varying jump distributions, beyond the compound Poisson processes.

The second difficulty lies in exact simulation of the sample path for general L\'evy processes. Unlike compound Poisson processes or random walks, the sample path of a general L\'evy process may not be exactly simulatable due to its infinite activities from the presence of either a Brownian motion or the infinitely many jumps within finite time intervals. Since many events $A_n$ that arise in applications can be characterized in terms of the extreme behavior of process $X$ within given time interval, one possible remedy is to simulate the extrema of the L\'evy process instead of the entire sample path. However, an explicit expression of the distribution of the extrema, or an exact simulation algorithm for extrema, is not available for L\'evy processes, except for a few specific cases such as spectrally one-sided processes \cite{michna2015distribution}\cite{chaumont2018short} or stable processes \cite{cazares2019exact}. 

In the current work, we address these issues by combining the stick-breaking approximation (SBA) idea for extrema of general L\'evy processes proposed in \cite{cazares2018geometrically} and the debiasing technique from \cite{rhee2015unbiased} with the mixture importance sampling from \cite{chen2019efficient}. The foundation of SBA is the detailed study of \cite{pitman2012convex} on the concave majorants of L\'evy processes, the distribution of which admits a Poisson-Dirichlet type of iterative structure, thus ensuring a geometrical convergence rate in SBA when estimating expectation of functionals on extrema of L\'evy processes. 
By studying the distributional properties of L\'evy processes, we show that our algorithm is strongly efficient for a broad class of L\'evy processes.

The rest of the paper is organized as follows. We provide preliminaries of the work in Section \ref{preliminaries}, and detail the algorithm in Section \ref{section_algorithm}. In Section \ref{section_algo_analysis} we establish a set of conditions under which the proposed algorithm is strongly efficient, and discuss the proper choice of parameters in the algorithm. In Section \ref{section_experiments} we demonstrate the  efficiency of the proposed importance sampling strategy with numerical experiments.

\section{PRELIMINARIES}
\label{preliminaries}

The importance sampling algorithm proposed in this paper builds upon the large deviations results for L\'evy processes with regularly varying increments. The related notions and results are introduced below. We use $(\mathbb{D},d)$ to denote the Skorokhod metric space of real-valued càdlàg functions with domain $[0,1]$. For any positive integer $l$, define
$$\mathbb{D}_l \delequal \{ \xi \in \mathbb{D}:\ \xi \text{ is a non-decreasing step function with $l$ jumps}, \xi(0) = 0 \}.$$
For $l = 0$, let $\mathbb{D}_0 = \{\textbf{0}\}$ where $\textbf{0}(t)= 0\ \forall t \in [0,1]$. Furthermore, for each $l\in\mathbb{N}^+$, we define
$\mathbb{D}_{< l} \delequal \bigcup_{j = 0}^{l - 1}\mathbb{D}_j$. 
Any L\'evy process $\{X(t): t\geq 0\}$ is characterized by its generating triplet $(c,\sigma^2,\nu)$, where $c \in \mathbb{R}$ is the drift parameter, $\sigma \geq 0$ is the magitude of the Brownian motion term in $X(t)$, and $\nu$ is the L\'evy measure of the process such that $\int (|x|^2\wedge 1) \nu(dx) < \infty$. See chapter 4 of \cite{sato1999levy} for details.

The heavy-tailed behavior of the positive jumps will be characterized by regular variation: recall that a Borel measurable function $\varphi: (0,\infty)\mapsto (0,\infty)$ is said to be regularly varying with index $\rho\in\mathbb{R}$ at $+\infty$ (denoted as $\varphi\in \text{RV}_\rho$) if for any $t>0$, $\lim_{x \rightarrow +\infty} \frac{\varphi(tx)}{\varphi(x)} = t^\rho$. For simplicity of the exposition, we focus on heavy-tailed behavior of positive jumps: In terms of the function $f(x) = \nu[x,\infty)$, we assume that $f \in RV_{-\alpha_+}$ with $\alpha_+ > 1$.

For any positive integer $n$, define the centered and scaled version of $X$ as 
$\Bar{X}_n(t) \delequal \frac{1}{n}X_n(t) - ct -\mu_1 t$
where $\mu_1 =  \int_{|x|\geq 1}x\nu(dx)$ and we assume $\mu_1 < \infty$.  For any $\beta > 0$, let $\nu_\beta$ be the measure concentrated on $(0,\infty)$ with $\nu_\beta(x,\infty) = x^{-\beta}$. For any positive integer $l$, use $\nu^l_\beta$ to denote the $l-$fold product measure of $\nu_\beta$ restricted onto $\{ y \in (0,\infty)^l:\ y_1 \geq y_2 \geq \cdots \geq y_l \}$, and define the measure
$C_{\beta}^l(\cdot) \delequal \mathbb{E}\Bigg[ \nu_\beta^l\big\{ y \in (0,\infty)^l:\ \sum_{j = 1}^l y_j \mathbbm{1}_{[U_j,1]}\in\cdot\big\} \Bigg]$
where $(U_j)_{j \geq 1}$ is an i.i.d. sequence of $\text{Unif}(0,1)$; while for $l = 0$, let $C^0_{\beta}$ be the Dirac measure on $\textbf{0}$. The following results describes the sample path large deviations for the corresponding scaled process $\Bar{X}_n$.

\begin{result}{(Theorem 3.1/3.4 of \cite{rhee2016sample})}\label{resultLevyLD} For any set $A$ that is Borel measurable in $\mathbb{D}$ and is bounded away from $\mathbb{D}_{<l^*}$ where $l^*\delequal \min\{l \in \mathbb{N}:\ \mathbb{D}_l \cap A \neq \emptyset \}$, we have
$$C^{l^*}_{\alpha_+}(A^\circ) \leq \liminf_{ n \rightarrow \infty }\frac{\mathbb{P}(\Bar{X}_n \in A)}{(n\nu[n,\infty) )^{l^*}}\leq \limsup_{ n \rightarrow \infty }\frac{\mathbb{P}(\Bar{X}_n \in A)}{(n\nu[n,\infty) )^{l^*}}\leq C^{l^*}_{\alpha_+}(A^-) $$
where $A^\circ,A^-$ are the interior and closure of $A$ respectively.
\end{result}{}

The rare events we concern in this paper are characterized by the extrema of the L\'evy process $X$. For any $t > 0$, we define the running supremum and infimum processes of $X$ as 
$\Bar{M}(t) = \sup_{s \in [0,t]}X(s)$.
Results in \cite{pitman2012convex} provide useful tools for studying $\bar{M}$ using the iterative structure of the concave majorant of L\'evy processes. 
Specifically, given any $t>0$ and a L\'evy process $X$ that is not a compound Poisson process with drift, the distribution of the pair $(\Bar{M}(t), X(t))$ admits the following expression
\begin{align}
    (\Bar{M}_t,X_t) \distequal ( \sum_{j \geq 1}(\xi_{j})^+,\sum_{j\geq 1}\xi_{j}  ) \label{SBA_preliminary}
\end{align}
where the symbol $\distequal$ denotes equivalence in distribution and $(\cdot)^+ \delequal \max\{\cdot, 0\}$, $\xi_{j}$'s are independent random variables such that $\xi_{j} \distequal X(l_j)$
where $(l_j)_{j \geq 1}$ is the stick-breaking sequence defined by $(U_j)_{j \geq 1}$, a sequence of i.i.d.\ $\text{Unif}(0,1)$, as follows:
$$l_1 = tU_1,\ \ l_j = U_j(t - l_1 - l_2 - \cdots l_{j-1})\ \forall j \geq 2.$$
A similar expression applies to the running infimum. See Theorem 1 in \cite{pitman2012convex} for details, and \cite{cazares2018geometrically} for the stick-breaking approximation (SBA) algorithm for efficient approximation of extrema of L\'evy processes. This probabilistic description of concave minorants has been adapted to the stick-breaking approximation (SBA) in \cite{cazares2018geometrically} that achieves geometric convergence in terms of the simulation of extrema of Levy processes. SBA is a major component of the algorithm proposed in this work, and the following result is frequently used to establish the rate of convergence when SBA is applied.

\begin{result}{(Lemma 10 in \cite{cazares2018geometrically})}\label{resultSBA} For a L\'evy process $X$ with generating triplet $(c,\sigma^2,\nu)$ that satisfies $I^1_+ = \int_{[1,\infty)}x\nu(dx)<\infty$, there exists a constant $C_X \in (0,\infty)$ such that for any $t > 0$:
$$\mathbb{E}\Bar{M}_t \leq C_X(t + \sqrt{t}).$$
\end{result}{}

To achieve unbiasedness for the proposed estimators, we apply the debiasing techniques used in \cite{rhee2015unbiased}:

\begin{result}{(Theorem 1 in \cite{rhee2015unbiased})}\label{resultDebias} Given a random variable $Y$ and a sequence of random variables $(Y_n)_{n \geq 0}$ such that $\lim_{n \rightarrow \infty}\mathbb{E}Y_n = \mathbb{E}Y$, and a positive integer-valued random variable $N$ with unbounded support such that $N$ is independent of $(Y_n)_{n\geq 0}$ and $Y$, if
$\sum_{n \geq 1} \mathbb{E}|Y_{n-1} - Y|^2 \Big/\mathbb{P}(N \geq n) < \infty,$
then for
$$Z = \sum_{n = 1}^N (Y_n - Y_{n - 1})\Big/\mathbb{P}(N \geq n), $$
(with the convention $Y_{-1} = 0$) $Z$ is an element of $L^2$, and
\begin{align*}
    \mathbb{E}Z = \mathbb{E}Y,\ \ \ \ \mathbb{E}Z^2 = \sum_{n \geq 0}\Bar{v}_n \Big/\mathbb{P}(N \geq n),
\end{align*}{}
where $\Bar{v}_n = \mathbb{E}|Y_{n - 1} - Y|^2 - \mathbb{E}|Y_n - Y|^2$.

\end{result}{}

The goal of the work is to propose an importance sampling algorithm for rare-event simulation of heavy-tailed L\'evy processes that achieves strong efficiency. Specifically, for a sequence of events $(A_n)_{n \geq 1}$ such that $\mathbb{P}(A_n)\rightarrow 0$ as $n\rightarrow \infty$, we say that a sequence of estimators $(L_n)_{n \geq 1}$ is unbiased and \textit{strongly efficient} if we have $\mathbb{E}L_n = \mathbb{P}(A_n)$ for any $n \geq 1$, and $\mathbb{E}L^2_n = \mathcal{O}(\mathbb{P}^2(A_n))$. Here, for two sequences of non-negative real numbers $(x_n)_{n \geq 1}$ and $(y_n)_{n \geq 1}$, we say $x_n = \mathcal{O}(y_n)$ if  $\limsup_{n\to\infty}\frac{x_n}{y_n} < \infty$. Besides, we write $x_n = o(y_n)$ for the two positive real sequences if $\lim_{n \rightarrow \infty}\frac{x_n}{y_n} = 0$.

\section{THE ALGORITHM}
\label{section_algorithm}

In this section, we describe the structure of the rare events we are interested in, and propose an importance-sampling algorithm for efficient estimation of their probability. For clarity, this section focuses on one running example which will be introduce shortly, and describes the algorithm tailored for the specific example. Nevertheless, it is worth mentioning that the principle underlying the algorithm proposed below enjoys greater flexibility and can be extended to more general cases. 

\subsection{The Rare Events $(A_n)_{n \geq 1}$ and the Process $X(t)$}
Define the set
\begin{align}
    A = \{\xi \in \mathbb{D}: \sup_{t \in [0,1]}\xi(t)\geq a; \sup_{t \in (0,1] }\xi(t) - \xi(t-) < b \}. \label{definition_of_A}
\end{align}{}
In other words, $\xi \in A$ if the supremum of $\xi$ has reached $a$, but no jump in $\xi$ is larger than $b$.
Furthermore, we make the following assumption about set $A$:
\begin{assumption}\label{assumption_setA}
$a,b>0,\ a/b\notin \mathbb{Z}.$
\end{assumption}
Consider $l^*\delequal \min\{l \in \mathbb{N}:\ \mathbb{D}_l \cap A \neq \emptyset \}$. In this case, we have $l^* = \ceil{a/b}$, and $l^* \geq 1$. Recall that for a step function $\xi$ to belong to set $A$, $\xi$ needs to have at least $l^*$ jumps. Moreover, it is easy to see that, under Assumption \ref{assumption_setA}, the set $A$ is bounded away from $\mathbb{D}_{<l^*}$, and $C^{l^*}_\beta(A^\circ) > 0$ for any $\beta > 0$.

We study a L\'evy process $\{X_t:t\geq 0\}$ with generating triplet $(c_X,\sigma^2,\nu)$. Since the case of compound Poisson processes were already treated in \cite{chen2019efficient}, we assume in this paper that $X$ is not a compound Poisson process with linear drift, which implies that either $\sigma > 0$ or $\nu(-1,1)=\infty$.  
Furthermore, we reiterate several assumptions: (1) $\int_{|x|>1}|x|\nu(dx) < \infty$ so $X_t \in L_1$ for any $t\geq 0$; (2) as for the heavy-tail behavior of the positive jumps, the function $f(x) = \nu[x,\infty)$ is regularly varying at $\infty$ with index $-\alpha_+ < -1$; (3) the drift coefficient $c_X$ is chosen specifically so that the process is already centered: $\mathbb{E}X_t = 0,\ \forall t > 0$.
In this case, the scaled and centered version of $X$ is $\bar{X}_n = \{X(nt)/n:\ t\in[0,1]\}$ for any $n \in \mathbb{Z}^+$.

Let $A_n \delequal \{\bar{X}_n \in A\}$. The goal is to propose an algorithm for estimating $\mathbb{P}(A_n)$. To achieve unbiasedness and strong efficiency of the algorithm, we need the following assumption regarding distributions of $X(t)$.
For a measure space $(\mathcal{X},\mathcal{F},\mu)$ and any $A \in \mathcal{F}$, denote the restriction of the measure $\mu$ on $A$ as
$\mu|_A( \cdot ) \delequal \mu( A \cap \cdot ).$ 
\begin{assumption}\label{assumption_Xdist}
For any $z_0 > 0$, there exist $C>0, \alpha > 0, \theta \in (0,1]$ such that for any $t > 0, z \geq z_0, x \in \mathbb{R}, \delta \in [0,1]$, we have
    \begin{align*}
        \mathbb{P}(X^{<z}(t) \in [x, x + \delta]) \leq \frac{C}{t^\alpha \wedge 1}\delta^\theta;
    \end{align*}
    where the process $X^{<z}$ is the L\'evy process with the generating triplet $(c_X,\sigma^2,\nu|_{(-\infty,z)})$. 
\end{assumption}
A process that has the same distribution as $X^{<z}$ can be obtained by removing all jumps larger than $z$ from $X$. Similarly, we define $X^{\geqslant z}$ as the compound Poisson process with the generating triplet $(0,0,\nu_{[z,\infty)})$, and $X^{\geqslant z}$ is understood as the compound Poisson process generated merely by all jumps larger than $z$ in $X$. Note that in Section 4, we show that Assumption 2 is a moderate condition.

\subsection{Importance Sampling Strategy and Construction of the Unbiased Estimator}\label{sectionIS}

Our algorithm builds on the construction of the importance sampling distribution in \cite{chen2019efficient}. Consider the rare event simulation problem for some fixed scaling level $n \in \mathbb{Z}^+$. For any $\gamma > 0$, define sets $B^\gamma_n \delequal \{\bar{X}_n \in B^\gamma\}$ where
$$B^\gamma \delequal \{ \xi \in \mathbb{D}: \#\{ t \in [0,1]: \xi(t) - \xi(t-) \geq \gamma \} \geq l^* \}; $$
namely, $\xi \in B^\gamma$ if and only if $\xi$ has at least $l^*$ jumps with size larger than $\gamma$. Now fix any $w \in (0,1)$, and define the following importance sampling distribution
\begin{align}
\mathbb{Q}(\cdot) = w\mathbb{P}(\cdot) + (1 - w) \mathbb{P}(\cdot | B^\gamma_n). \label{defQ}    
\end{align}

In the meantime, consider the following decomposition of $X$ at point $n\gamma$: $X = \widetilde{X} + J_n$ where the two independent processes $\widetilde{X} = X^{<n\gamma}, J_n = X^{\geqslant n\gamma}$ can be though of as the \textit{small-jump} and \textit{large-jump} processes of $X$, with generating triplets $(c_X,\sigma,\nu|_{(-\infty,n\gamma) })$ and $(0,0,\nu|_{[n\gamma,\infty)})$ respectively.

Now let us observe two facts: first, $\mathbb{Q}$ is absolutely continuous w.r.t. $\mathbb{P}$ and vise versa; second, using the decomposition above, we see that $X^{<n\gamma}$ admits the same marginal distribution under $\mathbb{P}$ and $\mathbb{Q}$, as $\mathbb{Q}$ only alters the distribution of the \textit{large-jump} process $J_n$. Therefore, to generate a sample path of $X$ under $\mathbb{Q}$, we can sample the large jump process $J_n$ from $\mathbb{Q}$, and then generate $\widetilde{X}$ under $\mathbb{P}$. To be more precise, we define set $E = \{ \xi \in \mathbb{D}:\ \sup_{t \in [0,1]}\xi(t) - \xi(t-) < b  \},$
and propose the following importance sampling estimator
\begin{align}
    L_n = Z_n(J_n)\mathbbm{1}_E( J_n/n ) \frac{d\mathbb{P}}{d\mathbb{Q}} = \frac{ Z_n(J_n)\mathbbm{1}_E( J_n/n )   }{ w + \frac{1 - w}{\mathbb{P}(B^\gamma_n)}\mathbbm{1}_{B^\gamma_n}(J_n)   } \label{defLn_IS}
\end{align}
where $J_n$ is sampled from $\mathbb{Q}$ and $Z_n$ is a stochastic function such that for any step function $\zeta$ on $[0,n]$,
$$\mathbb{E}Z_n(\zeta) = \mathbb{P}\Big( \sup_{t \in [0,n]}\widetilde{X}(t) + \zeta(t) \geq na  \Big).$$

Now it remains to describe: (a) the construction of $Z_n$; (b) the procedure of sampling $J_n$ from $\mathbb{Q}$ (in particular, sampling $J_n$ from $\mathbb{P}(\cdot|B^\gamma_n)$). For the first task, we combine the SBA algorithm with the debiasing technique as follows. To begin with, the nature of a jump process indicates the existence of some $k \in \{0,1,2,\cdots\}$ and sequences of real numbers $(z_i)_{i = 1}^k,(u_i)_{i = 1}^k$ with $u_i \in [0,n]$ and $(u_i)_{i=1}^k$ being distinct, such that $\zeta_k = \sum_{i=1}^k z_i\mathbbm{1}_{[u_i,n]}$. From now on we use the subscript $k$ to indicate the number of jumps in $\zeta$. Given the representation $\zeta_k = \sum_{i = 1}^k z_i\mathbbm{1}_{[u_i,n]}$, the interval $(0,n]$ can be partitioned into $\{ (u_i, u_{i+1}] \}_{i = 0}^k$ with the convention that $u_0 = 0, u_{k+1} = n$. For each $i = 0,1,\cdots,k$, we conduct the following stick-breaking procedure on $(u_i,u_{i+1}]$:
\begin{align}
    l^{(i)}_1 & = U^{(i)}_1(u_{i+1} - u_i); \label{defStickLength1} \\ 
    l^{(i)}_j & = U^{(i)}_j(u_{i+1} - u_i - l^{(i)}_1 - l^{(i)}_2 -\cdots - l^{(i)}_{j-1}) \ \ \forall j = 2,3,\cdots \label{defStickLength2}
\end{align}
where $(U^{(i)}_j)_{j \geq 1}$ is an i.i.d. sequence of $\text{Unif}(0,1)$. Next, for any given $0\leq i\leq k, j \geq 1$, independently sample $\xi^{(i)}_j \sim F_{\widetilde{X}}(\cdot,l^{(i)}_j)$ where we use $F_Y(\cdot,t)$ to denote the law of $Y_t$ for any L\'evy process $Y$. Let us define (for any $i = 0,1,\cdots,k$) $ \widetilde{M}^{(i)} \delequal \sum_{l = 0}^{i-1}\sum_{j \geq 1}\xi^{(l)}_j + \sum_{j \geq 1}(\xi^{(i)}_j)^+ $
with the convention that $\sum_{i = 0}^{-1}\cdot = 0$ and $(\cdot)^+ = \max\{0,\cdot\}$. Due to the coupling in \eqref{SBA_preliminary}, one can see that
$$\Big( \sup_{t \in (u_0,u_1]} \widetilde{X}_t,\sup_{t \in (u_1,u_2]} \widetilde{X}_t,\cdots,\sup_{t \in (u_k,u_{k+1}]} \widetilde{X}_t \Big)\distequal \Big( \widetilde{M}^{(0)},\widetilde{M}^{(1)},\cdots,\widetilde{M}^{(k)} \Big).$$

Recall our current task: unbiased estimation for expectation of the indicator random variable  
\begin{align*}
    Y_n^*(\zeta_k) = \mathbbm{1}\Big\{ \max_{i = 0,1,\cdots,k}\widetilde{M}^{(i)} + \zeta_k(u_i) \geq na  \Big\}
\end{align*}
given $\zeta_k$.
To apply the debiasing technique, the next step is to define a sequence of random variables $(Y_{n,m}(\zeta_k))_{m \geq 1}$ and approximate $Y^*_n(\zeta_k)$, where the subscript $m$ indicates the approximation level of SBA employed by $Y_{n,m}$. Specifically, for any $m \geq 0$, define
\begin{align*}
    \widetilde{M}^{(i)}_m = \sum_{l = 0}^{i - 1}\sum_{ j \geq 0 }\xi^{(l)}_j + \sum_{j = 1}^{ \ceil{ \log_2(n^2) } + m }( \xi^{(i)}_j )^+
\end{align*}
where $\ceil{x}$ denotes the smallest integer that is larger than or equal to $x$. 

Several remarks about the term $ \widetilde{M}^{(i)}_m$: (a) As an approximation to $\widetilde{M}^{(i)}\distequal \sup_{t \in (u_i,u_{i+1}]}\widetilde{X}_t$, $\widetilde{M}^{(i)}_m$ differs from $\widetilde{M}^{(i)}$ as it only inspects the increments of $\widetilde{X}$ on finitely many sticks; (b) Term $\ceil{ \log_2(n^2) }$ dictates that: as far as SBA is concerned, the algorithm always performs at least $\ceil{\log_2(n^2)}$ SBA steps at the scaling level $n$; this choice serves to ensure the strong efficiency of the algorithm, and would not increase the expected computational time significantly. 

Now, by defining $Y_{n,m}(\zeta_k) = \mathbbm{1}\Big\{ \max_{i = 0,1,\cdots,k}\widetilde{M}^{(i)}_m + \zeta_k(u_i) \geq na  \Big\}$
for any $m \geq 0$ and let $Y_{n,-1}(\cdot) = 0$, we construct the desired unbiased estimator as follows:
\begin{align}
    Z_n(\zeta_k) = \sum_{m = 0}^\tau \Big(Y_{n,m}(\zeta_k) - Y_{n,m-1}(\zeta_k) \Big)\Big/\mathbb{P}(\tau \geq m) \label{defZn_IS},
\end{align}
where the randomized truncation index $\tau$, independent of everything else, is chosen to be geometrically distributed with law $\mathbb{P}(\tau > m) = \rho^m$ for some $\rho \in (0,1)$ in our algorithm. Due to $\tau$ being finite almost surely, the number of $\xi^{(i)}_j$  we need to generate for evaluation of $Z_n(\zeta_k)$ is finite and depends on the value of $\tau$. The said parametrization will be justified in Section 4 as we see that $(L_n)_{n \geq 1}$ is strongly efficient.

\subsection{Sampling from $\mathbb{P}(\cdot|B^\gamma_n)$}

Below we revisit the problem of sampling the \textit{large-jump} process $J_n$ from the conditional distribution $\mathbb{P}(\cdot|B^\gamma_n)$, and propose Algorithm \ref{algoSampleJnFromQ}. The rationale of the algorithm can be made clear once we observe the following facts, and the argument therein is a direct application of point transform and augmentation for Poisson random measures; for details, see Chapter 5 of \cite{resnick2007heavy}.

First, to simulate $J_n$ (under the original law $\mathbb{P}$), it suffices to simulate a Poisson random measure $N_n$ on $[0,n]\times \mathbb{R}^+$ with intensity measure $\textbf{Leb}[0,n]\times \nu_n$ where $\nu_n(\cdot) = \nu\big( \cdot \cap [n\gamma,\infty) \big).$ The Poisson random measure $N_n$ admits the expression
$$N_n(\cdot)=\sum_{i = 1}^{\widetilde{N}_n}\mathbbm{1}\{(S_i,W_i) \in\ \cdot\ \}$$
where $\widetilde{N}_n\sim\text{Poisson}(n\cdot\nu[n\gamma,\infty))$ is the number of simulated points in $N_n$, $(S_i)_{i \geq 1}$ is an iid sequence of $\text{Unif}(0,n)$, and $(W_i)_{i \geq 1}$ is an iid sequence from the distribution $\nu_n(\cdot)/\nu_n[n\gamma,\infty)$; here we interpret $S_i$ as the arrival time of the $i-$th large jump, $W_i$ as its height, and $\widetilde{N}_n$ as the number of jumps in $J_n$ on $[0.n]$. Next, consider the simulation of a Poisson random measure with intensity measure $\nu_n$ using the inversion function:
$$Q^{\leftarrow}_n (y)\delequal{} \inf\{s > 0: \nu_n[s,\infty) < y\}.$$  
\begin{algorithm}[H]
  \caption{Efficient Estimation of $\mathbb{P}(A_n)$}\label{algoISnoARA}
  \begin{algorithmic}[1]
    \Require $w \in (0,1), \gamma > 0, \rho \in (0,1)$ 

    \If{$\text{Unif}(0,1) < w$} \Comment{Sample $J_n$ from $\mathbb{Q}$}
        \State Sample $J_n = \sum_{i = 1}^k z_i \mathbbm{1}_{[u_i,n]}$ from $\mathbb{P}$
    \Else
        \State Sample $J_n = \sum_{i = 1}^k z_i \mathbbm{1}_{[u_i,n]}$ from $\mathbb{P}(\cdot|B^\gamma_n)$ using Algorithm \ref{algoSampleJnFromQ}
    \EndIf
    \State Let $u_0 = 0, u_{k+1} = n$.

    \State Sample $\tau \sim \text{Geom}(\rho)$ \Comment{Decide Truncation Index $\tau$}

    \For{$i = 0,1,\cdots,k$} \Comment{Generate Stick Lengths, and Decide Increments} 
        \State Sample $U^{(i)}_1 \sim \text{Unif}(0,1)$. Let $l^{(i)}_1 = U^{(i)}_1(u_{i+1} - u_i)$
        \State Sample $\xi_{i,1}\sim F_{\widetilde{X}}(\cdot,l^{(i)}_1)$
        \For{$j = 2,3,\cdots,\ceil{\log_2(n^2)} + \tau$}
            \State Sample $U^{(i)}_j\sim \text{Unif}(0,1)$. Let $l^{(i)}_j = U^{(i)}_j(u_{i+1} - u_i - l^{(i)}_1 - l^{(i)}_2 - \cdots - l^{(i)}_{j - 1})$
            \State Sample $\xi_{i,j} \sim F_{\widetilde{X}}(\cdot,l^{(i)}_j)$
        \EndFor
        \State Let $l^{(i)}_{ \ceil{\log(n^2)} + \tau + 1 } = u_{i+1} - u_i - l^{(i)}_1 - l^{(i)}_2 - \cdots - l^{(i)}_{\ceil{\log(n^2)} + \tau}$
        \State Sample $\xi_{i,\ceil{\log_2(n^2)} + \tau + 1} \sim F_{\widetilde{X}}(\cdot,l^{(i)}_{\ceil{\log_2(n^2)}+\tau+1})$
    \EndFor

    \For{$m = 0,1,\cdots,\tau$} \Comment{Evaluate $Y_{n,m}$}
        \For{$i = 0,1,2,\cdots,k$}
            \State Let $\widetilde{M}^{(i)}_m = \sum_{l = 0}^{i-1}\sum_{j =1}^{ \ceil{\log_2(n^2)} + \tau + 1}\xi^{m}_{l,j} + \sum_{j =1}^{ \ceil{\log_2(n^2)} + \tau }(\xi^{m}_{i,j})^+   $
        \EndFor 
        \State Let $Y_{n,m} = \mathbbm{1}\big\{ \max_{i = 0,1,\cdots,k} \widetilde{M}^{(i)}_m + J_n(u_i)  \geq na  \big\}$
    \EndFor

\State Let $Z_n = Y_{n,0} + \sum_{m = 1}^\tau ( Y_{n,m} - Y_{n,m-1} )\big/ \rho^{m-1}$ \Comment{Return the Estimator $L_n$}
\If{ $ \max_{i =1,\cdots,k}z_i > b $ }
    \State \textbf{Return} $L_n = 0$.
\Else 
    \State Let $\lambda_n = n\nu[n\gamma,\infty),\ p_n = 1 - \sum_{l = 0}^{l^* - 1}e^{-\lambda_n}\frac{\lambda_n^l}{l!},\ I_n = \mathbbm{1}\{ J_n \in B^\gamma_n \}$
    \State \textbf{Return} $L_n = Z_n/(w + \frac{1-w}{p_n}I_n)$
\EndIf 
    
  \end{algorithmic}
\end{algorithm}

\begin{algorithm}[H]
  \caption{Simulation of $J_n$ under $\mathbb{P}(\cdot | B^\gamma_n)$}\label{algoSampleJnFromQ}
  \begin{algorithmic}[1]
    \Require $n \in \mathbb{N}, l^* \in \mathbb{N}, \gamma > 0$, the Lévy measure $\nu$.
\State Sample $k \sim \text{Poisson}( n\cdot\nu[n\gamma,\infty) )$ conditioned on $\{k \geq l^*\}$
\State Sample $\Gamma_1,\cdots,\Gamma_k \stackrel{\text{i.i.d.}}{\sim} \text{Unif}[0,\nu[n\gamma,\infty)]$
\State Sample $U_1,\cdots,U_k \stackrel{\text{i.i.d.}}{\sim} \text{Unif}[0,n]$
\State \textbf{Return } $J_n = \sum_{i = 1}^k Q^{\leftarrow}_n(\Gamma_i)\mathbbm{1}_{[U_i,n]}$
    
  \end{algorithmic}
\end{algorithm}

This inverse function has the property that
    $y \leq \nu_n[s,\infty) \Leftrightarrow Q^{\leftarrow}_n (y) \geq s.$
    Therefore, for iid Exponential (with rate $1$) random variables $\{E_i\}_{i \in \mathbb{Z}^+}$ and the corresponding running sum $\Gamma_i = \sum_{j = 1}^i E_j$, it is known that 
    $$\sum_{i:\Gamma_i \leq \nu_n[n\gamma,\infty)}\delta_{Q^{\leftarrow}_n (\Gamma_i)}$$
    is the desired Poisson random measure, where $\delta_x$ denotes Dirac measure at $x$. Now, by augmenting $\{Q^{\leftarrow}_n(\Gamma_i)\}_{i \geq 1}$ with uniformly distributed random marks on $[0,n]$, we have that
    $$N_n \distequal{} \sum_{\Gamma_i \leq \nu_n[n\gamma,\infty)}\delta_{\big(U_i,Q^{\leftarrow}_n (\Gamma_i)\big)}; \ \ J_n \distequal{}\sum_{\Gamma_i \leq \nu_n[n\gamma,\infty)}Q^{\leftarrow}_n (\Gamma_i)\mathbbm{1}_{[U_i,n]} $$
    where $(U_i)_{i \geq 1}$ is a sequence of iid $\text{Unif}(0,n)$ random variables that are independent of $\{\Gamma_i\}_{i \geq 1}$. Lastly, the condition $\bar{X}_n \in B^\gamma$ is equivalent to $\sup\{i:\Gamma_i \leq \nu_n[n\gamma,\infty)\} \geq l^*$. For any $k \geq l^*$, by further conditioning on the event $\big\{\sup\{i:\Gamma_i \leq \nu_n[n\gamma,\infty)\} = k\big\}$, the distribution of $(\Gamma_1,\cdots,\Gamma_k)$ is the same as that of the order statistics of $k$ iid random variables from $\text{Unif}[0,\nu_n[n\gamma,\infty)]$. With the difficulty of sampling from $\mathbb{P}(\cdot|B^\gamma_n)$ resolved, we yield an importance sampling strategy that is readily implementable, and we detail the steps in Algorithm \ref{algoISnoARA},

\section{ANALYSIS OF THE ALGORITHM}
\label{section_algo_analysis}

This section is devoted to theoretical aspects of the proposed algorithm. In Theorem \ref{  thmStrongEfficiency } we show that, under Assumptions \ref{assumption_setA} and \ref{assumption_Xdist}, the proposed algorithm is strongly efficient when parametrized properly. In Theorem \ref{thmDistX}, by analyzing the distributional properties of L\'evy processes, we show that Assumption \ref{assumption_Xdist} can be verified by inspecting the behavior of L\'evy measure $\nu$ when close to the origin, and Assumption \ref{assumption_Xdist} is satisfied by a broad class of L\'evy processes with infinite activities.

\subsection{Strong Efficiency of $(L_n)_{n \geq 1}$}
We state the main result regarding the efficiency of the importance sampling algorithm. See Appendix \ref{appendixProof1} for the proof.

\begin{theorem}\label{  thmStrongEfficiency }
Suppose that Assumption \ref{assumption_setA} and Assumption \ref{assumption_Xdist} are in force, and
 $\gamma$ (which characterizes the set $B^\gamma$) and $\rho$ (which determines the distribution of $\tau \sim \text{Geom}(\rho)$) are as follows.
\begin{itemize}
    \item Choose $\gamma\in(0,\frac{a - (l^*-1)b}{l^*})$ such that $\frac{a - (l^*-1)b}{\gamma}$ is not an integer.
    \item Let $\alpha,\theta$ be the values stated in Assumption \ref{assumption_Xdist}, and choose 
    $$\delta \in (1/\sqrt{2},1),\ \alpha_3 \in (0,\frac{\theta}{\alpha}),\ \alpha_4 \in (0, \frac{\theta}{2\alpha}),\ \alpha_2 \in (0, (\alpha_3/2)\wedge 1),\ \alpha_1 \in (0,\frac{\theta}{\alpha\alpha_2}).$$
   Pick $\rho$ such that 
   $$1 > \rho > \sqrt{\max\{ \delta^\alpha, \frac{1}{\delta\sqrt{2}}, \delta^{\theta\alpha_2 - \alpha\alpha_1}, \delta^{\theta - \alpha\alpha_3}, \delta^{-\alpha_2 + \frac{\alpha_3}{2}} \}}.$$
\end{itemize}
Then, $(L_n)_{n \geq 1}$ is unbiased and strongly efficient for $(A_n)_{n \geq 1}$; namely; 
$$\mathbb{E}^\mathbb{Q}[L_n] = \mathbb{P}(A_n),\ \ \ \mathbb{E}^\mathbb{Q}[L^2_n] = \mathcal{O}(\mathbb{P}^2(A_n)).$$
\end{theorem}{}

\subsection{Distributional Property of Small-Jump Processes $X^{<z}$}
Below we provide a sufficient condition for Assumption~\ref{assumption_Xdist}, and show that a broad class of L\'evy processes therefore can be addressed by the proposed algorithm. 
In particular, the conditions below verifies Assumption \ref{assumption_Xdist} with $\theta = 1$, which is equivalent to showing Lipschitz continuity of the law of $X^{<z}$.

First of all, if $\sigma > 0$, then for a fixed $\gamma_0 > 0$ and any $\gamma \geq \gamma_0$, we have the decomposition
$X^{<\gamma}(t) = \sigma B(t) + Y^{<\gamma}(t) $
where $B$ is a standard Brownian motion, $Y^{<\gamma}$ is a Levy process with generating triplet $(0,0,\nu|_{(-\gamma,\gamma)})$, and the two processes are independent. Now for any $x \in \mathbb{R}$ and $\delta \in (0,1)$, we have
\begin{align*}
    \mathbb{P}(X^{<\gamma}(t) \in [x,x+\delta])&=\int_{\mathbb{R}}\mathbb{P}(\sigma B(t) \in [x-y,x-y+\delta])\cdot\mathbb{P}(Y^{<\gamma}(t) = dy) \\
    &\leq\frac{1}{\sigma\sqrt{2\pi}}\cdot\frac{\delta}{\sqrt{t}}.
\end{align*}
Therefore, Assumption \ref{assumption_Xdist} holds with $\theta = 1, \alpha = 1/2.$ 
From now on, we focus on the case where $\sigma = 0$. 
In addition, we also assume that $\nu(\mathbb{R}) = \infty$, because otherwise $X$ is a compound Poisson process and this case has already been addressed by \cite{chen2019efficient}. 
We say that any measurable function $h:(0,\infty)\mapsto(0,\infty)$ is regularly varying at $0$ with index $\rho$ if, for $\varphi(x) = h(1/x)$, we have $\varphi \in \text{RV}_{-\rho}$. 
\begin{theorem}\label{thmDistX} For a fixed $\gamma_0 > 0$ and a L\'evy process $\{X(t): t\geq 0\}$ with generating triplet $(0,0,\nu)$, suppose that we have some Borel measure $\mu$ such that 
\begin{itemize}
    \item $(\nu - \mu)|_{(-\gamma_0,\gamma_0)}\text{ is a positive measure}$;
    \item the function $f:(0,\infty)\mapsto(0,\infty)$ defined as $f(x) = \mu\big( (-\infty,-x]\cup[x,\infty) \big)$ is regularly varying at 0 with index $-(\alpha + \epsilon)$ where $\alpha \in (0,2), \epsilon \in (0,2 - \alpha)$.
\end{itemize}
Then there exists some $C < \infty$ such that for any $\gamma \geq \gamma_0$
$$\norm{ f_{X^{<\gamma}(t)} }_\infty \leq \frac{C}{t^{1/\alpha}\wedge 1}\ \ \forall t > 0 $$
where $\{X^{<\gamma}(t):\ t>0\}$ is the L\'evy process with generating triplet $(0,0,\nu|_{(-\gamma,\gamma)})$ and $f_{X^{<\gamma}(t)}$ is the density of distribution of $X^{<\gamma}(t)$.
\end{theorem}

For the proof, see Appendix \ref{appendixProof2}. An immediate consequence is as follows.
Define a function $g(x) = \nu\big( (\infty,-x)\cup(x,\infty) \big)$. 
If $g$ is regularly varying at $0$ with index $\beta > 0$, then Assumption \ref{assumption_Xdist} holds, and the proposed algorithm is strongly efficient. Intuitively, since we are excluding the simpler cases where $\sigma > 0$ or  $\nu(\mathbb{R})<\infty$, we must have $\lim_{x\downarrow 0}g(x)=\infty$. 
As long as $g(\cdot)$ approaches $\infty$ at a faster rate than some $1/x^\beta$ with $\beta > 0$, Assumption \ref{assumption_Xdist} is valid. 

\subsection{Sketch of Proof for Theorem \ref{  thmStrongEfficiency }}

By performing a change of measure and plugging in the exact value of $d\mathbb{Q}/d\mathbb{P}$ (see \eqref{defQ}):
\begin{align}
   \nonumber \mathbb{E}^{\mathbb{Q}}[L^2_n] & = \int Z^2_{n}(J_n)\mathbbm{1}_{E}(J_n / n) \frac{d\mathbb{P}}{d\mathbb{Q}}\frac{d\mathbb{P}}{d\mathbb{Q}}d\mathbb{Q} = \int Z^2_{n}(J_n)\mathbbm{1}_{E}(J_n/n) \frac{d\mathbb{P}}{d\mathbb{Q}}d\mathbb{P} \\ \nonumber
    & = \int Z^2_{n}(J_n)\mathbbm{1}_{E \cap B^\gamma_n}(J_n/n)\frac{d\mathbb{P}}{d\mathbb{Q}}d\mathbb{P} + \int Z^2_{n}(J_n)\mathbbm{1}_{E \cap (B^\gamma_n)^c}(J_n/n)\frac{d\mathbb{P}}{d\mathbb{Q}}d\mathbb{P} 
\leq  
\frac{\mathbb{P}(B^\gamma_n)}{1-w}\mathbb{E}[Z_{n,1}^2] + \frac{1}{w}\mathbb{E}[Z_{n,2}^2], 
\end{align}
where $Z_{n,1} = Z_{n}(J_n)\mathbbm{1}_{E\cap B^\gamma}(J_n / n), Z_{n,2} = Z_{n}(J_n)\mathbbm{1}_{E \cap (B^\gamma)^c}(J_n/n)$. Using Result \ref{resultLevyLD}, we have $\mathbb{P}(B^\gamma_n) = \mathcal{O}(\mathbb{P}(A_n))$ as both $B^\gamma$ and $A$ are bounded away from $\mathbb{D}_{<l^*}$. Then strong efficiency follows immediately once we have
\begin{align}
    \mathbb{E}Z^2_{n,1} = \mathcal{O}(\mathbb{P}(A_n)); \label{proofGoal1} \\
    \mathbb{E}Z^2_{n,2} = \mathcal{O}(\mathbb{P}^2(A_n)). \label{proofGoal2}
\end{align}
Fix some notations: we use $\zeta$ denote a step function, and save the index $k$ to indicate the number of large jumps. For instance, the event $\{ J_n = \zeta_k \}$ is equivalent to the event that $J_n$ has $k$ jumps. Note that on this set, $J_n$ admits the representation
$J_n \distequal{} \zeta_k = \sum_{i = 1}^k z_i\mathbbm{1}_{[u_i,n]} $
where  $z_1,\cdots,z_k$ are i.i.d. samples from the distribution $\nu(\cdot \cap [n\gamma,\infty))\big/\nu[n\gamma,\infty)$, and $u_1 \leq u_2 \leq \cdots \leq u_k$ are order statistics of $k$ i.i.d. $\text{Unif}(0,n)$. Now note that 
\begin{align}
    \mathbb{E}Z^2_{n,1} \leq\sum_{k \geq l^*} \mathbb{E}[ Z^2_n(J_n)\ |\ J_n = \zeta_k ]\mathbb{P}(J_n\text{ has k jumps}) = \sum_{k \geq l^*} \mathbb{E}[ Z^2_n(\zeta_k)]e^{-\lambda_n}\lambda^k_n/k! \label{proofGoal1Again}
\end{align}
with $\lambda_n = n\nu[n\gamma,\infty)$. Therefore, to show \eqref{proofGoal1}, it suffices to show the existence of a constant $C$ such that $\mathbb{E}Z^2_n(\zeta_k) \leq Ck$ for any $k = 1,2,\cdots$. To see this, by plugging this bound into R.H.S. of \eqref{proofGoal1Again} we will get
\begin{align}
    \mathbb{E}Z^2_{n,1}\leq C\sum_{k \geq l^*}ke^{-\lambda_n}\lambda^k_n/k! \leq C \lambda_n^{l^*} \sum_{k \geq l^*} e^{-\lambda_n}\frac{\lambda_n^{k - l^*} }{(k-l^*)!} = C\cdot( n\nu[n\gamma,\infty) )^{l^*}
\end{align}
and \eqref{proofGoal1} follows immediately from large deviation principles (Result \ref{resultLevyLD}) and the fact that $\nu$ is regularly varying. To bound $\mathbb{E}Z^2_n(\zeta_k)$, recall that $Z_n$ is an unbiased estimator, so from Result \ref{resultDebias} we have $\mathbb{E}Z^2_n(\zeta_k) \leq \sum_{m \geq 0} \mathbb{P}\Big(Y_{n,m}(\zeta_k) \neq Y^*_n(\zeta_k) \Big)\Big/\mathbb{P}(\tau\geq m)$, where $Y_{n,m}$ and $Y^*_n$ are indicator functions defined in Section \ref{sectionIS}. Then it remains to bound $\mathbb{P}\Big(Y_{n,m}(\zeta_k) \neq Y^*_n(\zeta_k) \Big)$, the probability that the supreme of a (non-compound-Poisson) L\'evy process crossed a certain barrier while its $m$ step SBA estimation did not. 

To illustrate the idea, we henceforth focus on a simplified scenario. Fix a constant $c>0$ and recall that small-jump process $\widetilde{X}$ is a (non-compound-Poisson) L\'evy process with bounded jumps and satisfies Assumption \ref{assumption_Xdist}. Define $M = \sup_{0 \leq t \leq 1}\tilde{X}(t)$. Using the coupling in \eqref{SBA_preliminary}, we have $M \distequal{} \sum_{i \geq 1}(\xi_i)^+ $ where $(\xi_i \distequal{} \tilde{X}(l_i))_{i \geq 1}$ are independent conditioned on the stick length sequence $(l_i)_{i \geq 1}$ with $\sum_i l_i = 1$. Furthermore, we use $Y^* =\mathbbm{1}\{ M \geq c  \}$ to indicate whether the supreme of $\tilde{X}$ on [0,1] exceeds $c$, while $Y_m = \mathbbm{1}\{ M_m \delequal{} \sum_{i = 1}^m(\xi_i)^+ \geq c \}$ as its counterpart for the $m-$step SBA estimation. To bound $\mathbb{P}(Y^* \neq Y_m)$, fix some $\epsilon \in (0,c)$ and notice that if $\{Y^* \neq Y_m\}$ occurs, then so does at least one of the following two events: (a) $|M - M_m| \geq \epsilon$; (b) $M \in [c - \epsilon, c+\epsilon]$. For the former, the geometric convergence rate of SBA gives an effective bound (in particular, see Lemma 10 in \cite{cazares2018geometrically}). Now the proof hinges on bounding the latter, which boils down to analyzing probability of the form $\mathbb{P}(M_j \in [c,c+\delta])$ for $c,\delta>0$. 

To this end, recall that $M_j \distequal{} (\xi_1)^+ * (\xi_2)^+ * \cdots * (\xi_j)^+$ where $*$ is the convolution operator, and observe the following facts. First, convolution operation preserves the smoothness of any distribution involved; for instance $X*Y$ is continuous if either $X$ or $Y$ is continuous, and $X*Y$ has a density bounded by a constant $K$ if so does either $X$ or $Y$. Besides, given $(l_i)_{i \geq 1}$, the law of $\xi_i \distequal{} \tilde{X}(l_i)$ satisfies Assumption \ref{assumption_Xdist}, and as long as $\xi_i > 0$, we will have $\xi_i = (\xi_i)^+$ so the same smoothness property can be passed to $(\xi_i)^+$, hence $M_j$. To apply these facts, fix some $h>0$ as a threshold value and note that: if there exists $i = 1,\cdots,j$ such that $l_i > 0$ and $\xi_i > 0$, then by further conditioning on this event we can use Assumption \ref{assumption_Xdist} to provide a bound with $t \geq h$; otherwise, the supreme $M$ is equal to sum of increments only on sticks shorter than $h$, the total length of which is less than $jh$. If $h$ (and hence $jh$) is indeed a small value, then it is unlikely that L\'evy process $\tilde{X}$ reached the barrier $c$ within such a short period of time $jh$ (again, see Lemma 10 in \cite{cazares2018geometrically}), let along crossing the barrier $c$ and staying in $[c,c+\delta]$. By carefully choosing $j$ and $h$, we can establish a useful upper bound for $\mathbb{P}(M \in [c,c+\delta])$, and eventually for $\mathbb{E}Z^2_n(\zeta_k)$.

The argument for \eqref{proofGoal2} will be analogous, except that we need to notice the following fact: since $Y^*_n \geq Y_{n,m}$, for $Y^*_n \neq Y_{n,m}$ to occur we need to at least ensure $Y^*_n = 1$, which is equivalent to the condition $\{\bar{X}_n \in A\}$. Combining this with Hölder's inequality when using Result \ref{resultDebias} and the bound $\mathbb{E}Z^2_n(\zeta_k) \leq Ck$ above, we will have
$\mathbb{E}Z^2_{n,2} \leq C_1\sqrt{\mathbb{P}( \bar{X}_n \in A\cap(B^\gamma)^c)}$
where $C_1 < \infty$ is some constant. Thus, by picking $\gamma$ small enough and invoking large deviation principles (Result \ref{resultLevyLD}) again, we then have \eqref{proofGoal2} and conclude the proof.

\section{SIMULATION EXPERIMENTS}
\label{section_experiments}

In this section, we apply the proposed importance sampling strategy in Algorithm \ref{algoISnoARA} to the following setting and use numerical experiments to demonstrate: (1) the performance of the importance sampling estimator under different scaling factor $n$ and different tail distributions; (2) the efficiency of the algorithm when compared to crude Monte-Carlo methods. 

Consider a L\'evy process
$X(t) = B(t) + \sum_{i=1}^{N(t)}W_i $
where $B(t)$ is the standard Brownian motion, $N$ is a Poisson process with arrival rate $\lambda = 0.1$, and $\{W_i\}_{i \geq 1}$ is a sequence of i.i.d. samples from Pareto distribution with
$$\mathbb{P}( W_1 > x ) = \frac{1}{\max\{x,1\}^\alpha}  $$
where the tail index $\alpha > 1$. For each $n \geq 1$, define the scaled process $X_n(t) = \frac{X(nt)}{n}$, and we are interested in the probability of the event $A_n = \{X_n \in A\}$ where
$$A = \{ \xi \in \mathbb{D}: \sup_{t \in [0,1]} \xi(t) - \xi(t-) < b, \ \sup_{t \in [0,1]}\xi(t) \geq a  \} $$
with $a = 2, b = 1.15$. As stressed above, we aim to showcase the performance of the importance sampling estimator under different $n$ and $\alpha$. Specifically, in our experiments we use $\alpha = 1.45, 1.6, 1.75$, and $n = 1000, 2000, \cdots, 10000$. To quantify the efficiency of an estimator, we report the \textit{relative error}: the ratio between the standard deviation estimated by all samples of the estimator and the estimated mean.

In terms of the specifications the of experiments, for the importance sampling estimator we use $\gamma = 0.2, w = 0.05, \rho = 0.95$ (and note that $l^* = 2$ in this case). For each $\alpha \in \{1.45,1.6,1.75\}$ and $n \in \{1000,2000,\cdots,10000\}$ we obtain 500,000 independent samples. We also compare the efficiency of the importance sampling estimator against the crude Monte-Carlo methods. As for the number of simulation trials we run for each $\alpha$ and $n$ for crude Monte-Carlo estimator, we ensure that at least $64/\hat{p}_{\alpha,n}$ samples are obtained where $\hat{p}_{\alpha,n}$ is the estimated value of $\mathbb{P}(A_n)$ using Algorithm \ref{algoISnoARA} as described above

Highlights of the experiment results are summarized in Table \ref{tab: IS relative error} and Figure \ref{fig: reinsurance}. In Table \ref{tab: IS relative error}, we see that, for a fixed $\alpha$, the relative error of the importance sampling estimator stays at a constant level regardless of how large $n$ is. This is as expected in view of the strong efficiency of the estimator established in Theorem \ref{  thmStrongEfficiency }. Therefore, in a rare-event simulation setting where the goal is to achieve a certain level of standard error, the number of samples required for Algorithm \ref{algoISnoARA} is upper bounded and does not scale with $n$.

Figure \ref{fig: reinsurance} illustrates the relative error of both methods, thus demonstrating the benefit of the proposed importance sampling strategy for rare-event simulation: for crude Monte-Carlo scheme the relative error scales polynomialy with $n$ (to be precise, roughly $\mathcal{O}(n^{l^*(\alpha - 1)})$), which contrasts the nearly constant relative error of the Algorithm \ref{algoISnoARA}. Now recall that: for both method, the expected cost to generate one sample is $\mathcal{O}(n)$ (which is the order of expected number of jumps to be simulated), so the proposed importance sampling method always outperforms crude Monte-Carlo scheme in terms of computational cost, assuming again the goal is to achieve a fixed level of standard error. Therefore, when the algorithm is properly parametrized, the larger the scaling factor $n$ is (namely, the rarer the events $A_n$ are) the more efficient and favorable our importance sampling estimator is when compared against vanilla Monte-Carlo approach. 
\begin{figure}[htb]
{
\centering
\includegraphics[width=0.7\textwidth]{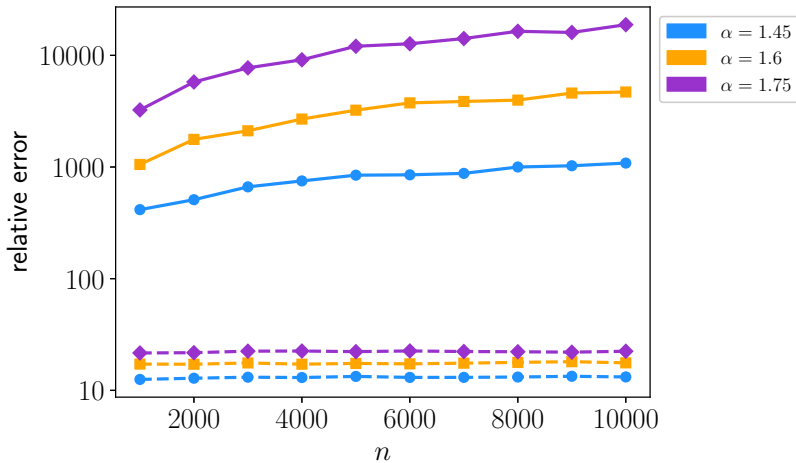}
\caption{Comparison of relative errors between the proposed importance sampling estimator and crude Monte Carlo estimator. Solid lines: Crude Monte-Carlo estimator; Dashed lines: Importance-sampling estimator. \label{fig: reinsurance}}
}
\end{figure}

\begin{table}[htb]
\centering
\caption{Rare-event simulation results using Algorithm \ref{algoISnoARA}. First row: estimated probability of $\mathbb{P}(A_n)$; Second row: the relative error.}\label{tab: IS relative error}
\begin{tabular}{c c c c c c} 
 \hline
 n & 2000 & 4000 & 6000 & 8000 & 10000 \\
 \hline
 \multirow{2}{4em}{$\alpha = 1.45$} & $3.53\times 10^{-6}$ & $1.85\times 10^{-6}$ & $1.28\times 10^{-6}$ & $9.76\times 10^{-7}$ & $7.96\times 10^{-7}$ \\
 & $12.84$ & $13.02$ & $13.06$ & $13.16$ & $13.19$ \\
 \hline
 \multirow{2}{4em}{$\alpha = 1.6$} & $3.34\times 10^{-7}$ & $1.45\times 10^{-7}$ & $8.84\times 10^{-8}$ & $5.89\times 10^{-8}$ & $4.60\times 10^{-8}$ \\
 & 17.13 & 17.16 & 17.26 & 17.80 & 17.63 \\
 \hline
\multirow{2}{4em}{$\alpha = 1.75$} & $3.46\times 10^{-8}$ & $1.14\times 10^{-8}$ & $6.21\times 10^{-9}$ & $4.17\times 10^{-9}$ & $2.92\times 10^{-9}$ \\
 & 21.74 & 22.50 & 22.53 & 22.16 & 22.40 \\
 \hline
\end{tabular}
\end{table}

\section{CONCLUSIONS}
We have proposed a strongly efficient importance sampling algorithm for rare-event simulation of L\'evy processes with heavy-tailed jump measures, where the events are triggered by multiple jumps and the L\'evy processes possess infinite activities. The algorithm outperforms crude Monte-Carlo method in terms of computational cost. In our future works, we aim to extend the current framework to more abstract events and a more general class of processes.

\footnotesize


\bibliographystyle{unsrt}
\bibliography{bib_appendix}

\pagebreak
\appendix
\section{Proof of Theorem \ref{  thmStrongEfficiency }} \label{appendixProof1}

\begin{proof}

Let us first fix some notations and choice of parameters. Let $\alpha,\theta$ be the values stated in Assumption \ref{assumption_Xdist}. Choose
\begin{align}
    \alpha_3 \in (0,\frac{\theta}{\alpha}),\ \alpha_4 \in (0, \frac{\theta}{2\alpha}).\label{proofChooseAlpha34}
\end{align}
 Next, fix 
 \begin{align}
     \alpha_2 \in (0, (\alpha_3/2)\wedge 1),\ \alpha_1 \in (0,\frac{\theta}{\alpha\alpha_2}). \label{proofChoose12}
 \end{align}
 Then, fix 
 \begin{align}
     \delta \in (1/\sqrt{2},1) \label{proofChooseDelta}
 \end{align}
 Since we require $\alpha_2$ to be strictly less than $1$, there exists some integer $\Bar{m}$ such that for any $m \geq \Bar{m}$, we have
 \begin{align}
     \delta^{m\alpha_2} - \delta^{m} \geq \frac{\delta^{m\alpha_2}}{2}. \label{proofChooseMBar}
 \end{align}
 Fix such $\bar{m}$. Based on all previous choices, we are able to choose $\rho_0,\rho \in (0,1)$ such that
\begin{align}
    \rho_0 & > \max\{ \delta^\alpha, \frac{1}{\delta\sqrt{2}}, \delta^{\theta\alpha_2 - \alpha\alpha_1}, \delta^{\theta - \alpha\alpha_3}, \delta^{-\alpha_2 + \frac{\alpha_3}{2}} \}, \label{proofchooseRho0} \\
    \rho & > \sqrt{\rho_0}, \label{proofchooseRho}
\end{align}
 Lastly, fix
 \begin{align}
     \gamma \in \Big(0, \frac{ a - (l^* - 1)b }{3l^*}\wedge b\Big) \label{proofChooseGamma}
 \end{align}
such that $(a - (l^* - 1)b)/\gamma$ is not an integer.
 The analysis of the efficiency of the importance sampling estimator starts by changing the measure from $\mathbb{Q}$ back to $\mathbb{P}$, and using the exact form of $d\mathbb{Q}/d\mathbb{P}$ (see \eqref{defQ}):
\begin{align}
   \nonumber \mathbb{E}^{\mathbb{Q}}[L^2_n] & = \int Z^2_{n}(J_n)\mathbbm{1}_{E}(J_n / n) \frac{d\mathbb{P}}{d\mathbb{Q}}\frac{d\mathbb{P}}{d\mathbb{Q}}d\mathbb{Q} = \int Z^2_{n}(J_n)\mathbbm{1}_{E}(J_n/n) \frac{d\mathbb{P}}{d\mathbb{Q}}d\mathbb{P} \\ \nonumber
    & = \int Z^2_{n}(J_n)\mathbbm{1}_{E \cap B^\gamma_n}(J_n/n)\frac{d\mathbb{P}}{d\mathbb{Q}}d\mathbb{P} + \int Z^2_{n}(J_n)\mathbbm{1}_{E \cap (B^\gamma_n)^c}(J_n/n)\frac{d\mathbb{P}}{d\mathbb{Q}}d\mathbb{P} \\
    &\leq  \frac{\mathbb{P}(B^\gamma_n)}{1-w}\mathbb{E}[Z_{n,1}^2] + \frac{1}{w}\mathbb{E}[Z_{n,2}^2], \label{proofFirstIneq} 
\end{align}
where $Z_{n,1} = Z_{n}(J_n)\mathbbm{1}_{E\cap B^\gamma}(J_n / n), Z_{n,2} = Z_{n}(J_n)\mathbbm{1}_{E \cap (B^\gamma)^c}(J_n/n)$. Using Result \ref{resultLevyLD}, we have $\mathbb{P}(B^\gamma_n) = \mathcal{O}(\mathbb{P}(A_n))$ as both $B^\gamma$ and $A$ are bounded away from $\mathbb{D}_{<l^*}$. The strong efficiency then follows immediately if we could establish
\begin{align}
    \mathbb{E}Z^2_{n,1} = \mathcal{O}(\mathbb{P}(A_n)); \label{proofGoal1} \\
    \mathbb{E}Z^2_{n,2} = \mathcal{O}(\mathbb{P}^2(A_n)). \label{proofGoal2}
\end{align}{}

The first half of the proof is devoted to establishing \eqref{proofGoal1}. By conditioning on $J_n$, we have
\begin{align}
    \mathbb{E}Z^2_{n,1} & = \mathbb{E}\Big[ \mathbb{E}[ Z^2_n(J_n)\mathbbm{1}_{E\cap B^\gamma}(J_n/n)\ |\ J_n] \Big] \nonumber \\
    & \leq \mathbb{E}\Big[ \mathbb{E}[ Z^2_n(J_n)\mathbbm{1}_{B^\gamma_n}(J_n / n)\ |\ J_n] \Big]. \label{proofZ1starting}
\end{align}{}
From now on, let $\lambda_n = n\nu[n\gamma,\infty)$. Fix the integer $k \geq l^*$. The proof proceeds by further conditioning on the event that there are $k$ jumps in $J_n$: in this case, $J_n$ is a compound Poisson process with arrival rate $\frac{\lambda_n}{n} = \nu[n\gamma,\infty)$ and jumps sampled from the distribution $\nu^{\text{normalized}}_n(\cdot) = \frac{\nu(\cdot \cap [n\gamma,\infty))}{\nu[n\gamma,\infty)}$, and we have the following equivalence (in distribution)
\begin{align}
    J_n \distequal{} \zeta_k \delequal{} \sum_{i = 1}^k z_i \mathbbm{1}_{[u_i,n]} \label{proofZetaRepresentation}
\end{align}{}
where $z_1,\cdots,z_k$ are i.i.d. samples from $\nu^{\text{normalized}}_n$, and $u_1 \leq u_2 \leq \cdots \leq u_k$ are order statistics of $k$ i.i.d. samples from $Unif(0,n)$. Note that
\begin{align}
    \mathbb{E}[ Z^2_{n}(J_n)  |\ J_n = \zeta_k ] & = \mathbb{E}Z^2_n(\zeta_k) \nonumber \\
    & \leq \sum_{ m \geq 0} \frac{ \mathbb{P}( Y_{n, m-1 }(\zeta_k) \neq Y^*_n(\zeta_k)  ) }{ \mathbb{P}(\tau \geq m) }, \label{proofIneqUnbiasEstL2withKjumps} 
\end{align}{}
where the inequality follows from Result \ref{resultSBA}, the construction of the unbiased estimator $Z_n$, and the fact that $Y^*_n$ or $Y_{n,m}$ only take value in $\{0,1\}$.

For any $m=0,1,2,\cdots$, consider the following decomposition of events based on $u_1$, the arrival time of the first large jump: (see Section \ref{sectionIS} for definitions of $Y^*_n, Y_{n,m}$)
\begin{align}
    \mathbb{P}( Y_{n, m} (\zeta_k) \neq Y^*_n(\zeta_k)  ) & = \mathbb{P}( Y_{n, m}(\zeta_k) \neq Y^*_n(\zeta_k),\ u_1 < n\delta^{m\alpha_1}  )\nonumber \\ & + \mathbb{P}( Y_{n, m}(\zeta_k) \neq Y^*_n(\zeta_k),\ u_1 \geq n\delta^{m\alpha_1}  )\nonumber \\
    & \leq \mathbb{P}(u_1 < n\delta^{m\alpha_1}) + \mathbb{P}( Y_{n, m}(\zeta_k) \neq Y^*_n(\zeta_k),\ u_1 \geq n\delta^{m\alpha_1}  ). \nonumber
\end{align}{}
The distribution of $u_1,\cdots,u_k$ implies
\begin{align}
    \mathbb{P}(u_1 < n\delta^{m\alpha_1}) \leq \sum_{i = 1}^k \mathbb{P}( Unif(0,n) < n\delta^{m\alpha_1}) = k\delta^{m\alpha_1} = k\cdot o(\rho^m) \label{proofFirstArrivalTimeBound}
\end{align}{}
due to \eqref{proofchooseRho}. As for the other term, for any $m \geq 0, i = 0,1,\cdots,k$, let us define
\begin{align*}
    Y_{n,m}^{(i)}(\zeta_k) & \delequal \mathbbm{1}\{\widetilde{M}^{(i)}_m(\zeta_k) + \sum_{j = 1}^i z_j  \geq na   \}, \\
    Y^{*,(i)}_n(\zeta_k) & \delequal  \mathbbm{1}\{\widetilde{M}^{(i)} + \sum_{j = 1}^i z_j  \geq na   \}.
\end{align*}
With the definitions above, we can rewrite $Y_n^*$ and $Y_{n,m}$ as 
\begin{align*}
    Y_n^*(\zeta_k) & = \max_{i = 0,1,\cdots,k} Y^{*,(i)}_{n}(\zeta_k) \\
    Y_{n,m}(\zeta_k) & = \max_{i = 0,1,\cdots,k} Y^{(i)}_{n,m}(\zeta_k)
\end{align*}{}

In other words, $Y^{*,(i)}_n(\zeta_k)$ detects whether the real sample path of the process $\widetilde{X} + \zeta_k$ has reached $na$ on $(u_{i},u_{i+1}]$, while $Y_{n,m}^{(i)}(\zeta_k)$ is a similarly defined indicator function but for the approximation at level $m$. Now one can see that
\begin{align}
    \mathbb{P}( Y_{n, m}(\zeta_k) \neq Y^*_n(\zeta_k),\ u_1 \geq n\delta^{m\alpha_1}  ) & \leq \sum_{i = 0}^k\mathbb{P}\Big(Y_{n,m}^{(i)}(\zeta_k) \neq  Y^{*,(i)}_n(\zeta_k), \ u_1 \geq n\delta^{m\alpha_1} \Big). \label{proofIneqFixedI}
\end{align}{}
To establish an upper bound for \eqref{proofIneqFixedI}, we consider two cases separately: $i = 0$; or $i \geq  1$. For the case where $i \geq 1$, let us observe the following facts. First, recall that
\begin{align*}
    \widetilde{M}^{(i)} & = \sum_{l = 0}^{i - 1}\sum_{j \geq 1}\xi^{(l)}_j + \sum_{j \geq 1}(\xi^{(i)}_j)^+ \\
    \widetilde{M}^{(i)}_m & = \sum_{l = 0}^{i-1}\sum_{j \geq 1]}\xi^{(l)}_j + \sum_{j = 1}^{\ceil{\log(n^2} + m}(\xi^{(i)}_j)^+
\end{align*}{}
where the coupling between $(\xi^{(i)}_j)_{i,j}$ are described in Section \ref{sectionIS}. Due to the equality $| \widetilde{M}^{(i)} -  \widetilde{M}^{(i)}_m| = \sum_{j>\ceil{\log(n^2} + m}(\xi^{(i)}_j)^+$, we know that: if both $\{Y_{n, m}(\zeta_k) \neq Y^*_n(\zeta_k)\}$ and $\{ \sum_{j>\ceil{\log(n^2} + m}(\xi^{(i)}_j)^+ < \delta^m/\sqrt{n} \}$ occur, then event
$$\{ \widetilde{M}^{(i)} + \zeta_k(u_i) \in [na, na + \frac{\delta^m}{\sqrt{n}}] \} $$
must occur. Now, zooming in on event
$\{ \widetilde{M}^{(i)}_m + \zeta_k(u_i) \in [na - \delta^{m}/\sqrt{n}, na+\delta^m/\sqrt{n}]  \}$, since we know that
\begin{align*}
    \widetilde{M}^{(i)}_m \distequal{} \widetilde{X}^\prime_{u_i} + \sum_{j = 1 }^{ \ceil{\log_2(n^2)} + m } ( \xi^{ (i),\kappa_{n,m} }_j)^+
\end{align*}{}
where $\widetilde{X}^\prime_{u_i}$ is an independent copy of $\widetilde{X}_{u_i}$ and is independent of $( \xi^{ (i),\kappa_{n,m} }_j)_{j \geq 1}$, we can decide the location of $\widetilde{M}^{(i)}_m + \zeta_k(u_i)$ by first considering the value of $\widetilde{X}^\prime_{u_i}$ (thus deciding $\widetilde{X}^\prime_{u_i} + \zeta_k(u_i)$ ), and then considering $\sum_{j = 1 }^{ \ceil{\log_2(n^2)} + m } ( \xi^{ (i),\kappa_{n,m} }_j)^+$. This convolutional structure gives the following bound:
\begin{align*}
    & \mathbb{P}( \widetilde{M}^{(i)}_m + \zeta_k(u_i) \in [na - \delta^{m}/\sqrt{n}, na+\delta^m/\sqrt{n}]  ) \\
    \leq & \mathbb{P}( \widetilde{X}^\prime_{u_i} + \zeta_k(u_i) \in [na - \delta^{m\alpha_2}, na + \delta^{m\alpha_2}]  ) \\ 
    + &  \int_{ \mathbb{R}\symbol{92}[na - \delta^{m\alpha_2}, na + \delta^{m\alpha_2}]  } \mathbb{P}\Big( \sum_{j = 1 }^{ \ceil{\log_2(n^2)} + m } ( \xi^{ (i) }_j)^+ \in [ na - x - \frac{\delta^m}{\sqrt{n}},na - x + \frac{\delta^m}{\sqrt{n}} ]  \Big) \mathbb{P}( \widetilde{X}^\prime_{u_i} + \zeta_k(u_i) = dx  ) \\
    = &  \mathbb{P}( \widetilde{X}^\prime_{u_i} + \zeta_k(u_i) \in [na - \delta^{m\alpha_2}, na + \delta^{m\alpha_2}]  ) \\
    + & \int_{ (-\infty,na - \delta^{m\alpha_2})  } \mathbb{P}\Big( \sum_{j = 1 }^{ \ceil{\log_2(n^2)} + m } ( \xi^{ (i) }_j)^+ \in [ na - x - \frac{\delta^m}{\sqrt{n}},na - x + \frac{\delta^m}{\sqrt{n}} ]  \Big) \mathbb{P}( \widetilde{X}^\prime_{u_i} + \zeta_k(u_i) = dx  ).
\end{align*}{}
Overall, the observations above allow us to perform the following decomposition of events:
\begin{align}
    & \mathbb{P}\Big(Y_{n,m}^{(i)}(\zeta_k) \neq  Y^{*,(i)}_n(\zeta_k), \ u_1 \geq n\delta^{m\alpha_1} \Big) \label{proofCrucialBound} \\
\leq & \mathbb{P}\Big( \sum_{j > \ceil{\log_2(n^2)} + m} (\xi_{j}^{(i)} )^+ >  \frac{\delta^m}{\sqrt{n}} \Big) \label{proofTailOfSBA} \\
    + & \mathbb{P}\Big( \widetilde{X}(u_i) + \zeta_k(u_{i}) \in [na - \delta^{m\alpha_2},na + \delta^{m\alpha_2}],\ u_1 \geq n\delta^{m\alpha_1} \Big) \label{proofLocationOfX} \\
    + & \int_{ (-\infty,na - \delta^{m\alpha_2})  } \mathbb{P}\Big( \sum_{j = 1 }^{ \ceil{\log_2(n^2)} + m } ( \xi^{ (i) }_j)^+ \in [ na - x - \frac{\delta^m}{\sqrt{n}},na - x + \frac{\delta^m}{\sqrt{n}} ]  \Big) \mathbb{P}( \widetilde{X}_{u_i} + \zeta_k(u_i) = dx  ) \label{proofSBAConvolutionBound}
\end{align}
Now we focus on providing a bound for each term involved in the inequality above. For term \eqref{proofTailOfSBA}, using Result \ref{resultSBA} one can see the existence of some constant $C_X$ (which only depends on the generating triplet of $X$) such that: when conditioning on the length of the remaining stick $\sum_{j > \ceil{\log_2(n^2)} + md}l_j = t$, we have
\begin{align}
   & \mathbb{P}\Big( \sum_{j > \ceil{\log_2(n^2)} + m} (\xi_{j}^{(i)})^+ >  \frac{\delta^m}{3\sqrt{n}} \ | \ \sum_{j > \ceil{\log_2(n^2)} + m}l^{(i)}_j = t \Big) \nonumber \\
   \leq & 3C_X \frac{ \sqrt{n}( t + \sqrt{t}) }{\delta^m}. \label{proofFirstUseSBAMomentBound}
\end{align}{}
Therefore, unconditionally, due to $\sum_{j > \ceil{\log_2(n^2)} + m}l^{(i)}_j \distequal{} (u_{i+1} - u_i)\prod_{j=1}^{\ceil{\log_2(n^2)} + m}U_j$ where $(U_j)_{j \geq 1}$ is a sequence of i.i.d. samples of $Unif(0,1)$, the bound $u_{i+1} - u_i \leq n$, and the fact (due to Jensen's) that $\mathbb{E}\sqrt{W} \leq \sqrt{\mathbb{E}W}$ for any nonnegative random variable $W$, we have
\begin{align}
    &  \mathbb{P}\Big( \sum_{j > \ceil{\log_2(n^2)} + m} (\xi_{j}^{(i)})^+ >  \frac{\delta^m}{3\sqrt{n}} \Big) \leq 3C_X \frac{\sqrt{n}( \frac{n}{n^2 2^{m}} + \sqrt{\frac{n}{n^2 2^{m}}} )}{\delta^m}
    \leq \frac{ 6C_X }{(\sqrt{2} \delta)^m} = o(\rho^m) \label{proofEndTailofSBA}
\end{align}{}
due to our choice of $\rho$ in \eqref{proofchooseRho}.

As for term \eqref{proofLocationOfX}, representation of the step function $\zeta_k$ in \eqref{proofZetaRepresentation} gives $\zeta_k(u_i) = z_1 + \cdots + z_i$, and we have
\begin{align}
    & \mathbb{P}\Big( \widetilde{X}(u_i) + \zeta_k(u_{i}) \in [na - \delta^{m\alpha_2},na + \delta^{m\alpha_2}],\ u_1 \geq n\delta^{m\alpha_1} \Big) \nonumber \\
    \leq & \int_\mathbb{R} \mathbb{P}\Big( \widetilde{X}(u_i)\in [na - x - \delta^{m\alpha_2},na - x + \delta^{m\alpha_2}] \ | \ u_1 \geq n\delta^{m \alpha_1}  \Big) \mathbb{P}(\sum_{l = 1}^i z_l = dx) \nonumber \\
    = & \int_\mathbb{R} \int_{[n\delta^{m\alpha_1} ,n]^2} \mathbb{P}\Big( \widetilde{X}(u_i)(t)\in [na - x - \delta^{m\alpha_2},na - x + \delta^{m\alpha_2}]\Big)\mathbb{P}(U_1 = ds,U_i = dt) \mathbb{P}(\sum_{l = 1}^i z_l = dx) \nonumber
\end{align}{}
Due to Assumption \ref{assumption_Xdist}, for any $y \in \mathbb{R}, t \geq n\delta^{m\alpha_1} \geq \delta^{m\alpha_1}$, we have
$$\mathbb{P}(\widetilde{X}(t) \in [y, y + 2\delta^{m\alpha_1}] )  = \mathbb{P}(X^{\leq n\gamma}(t) \in [y, y + 2\delta^{m\alpha_1}] ) \leq \frac{2^\theta C (\delta^{\theta \alpha_2})^m }{ (\delta^{\alpha \alpha_1})^m }$$
where the constants $C,\alpha,\theta$ are given by Assumption \ref{assumption_Xdist}. By plugging this bound back into the integral above, we have
\begin{align}
     \mathbb{P}\Big( \widetilde{X}(u_i) + \zeta_k(u_{i}) \in [na - \delta^{m\alpha_2},na + \delta^{m\alpha_2}],\ u_1 \geq n\delta^{m\alpha_1} \Big) \leq 2^\theta C\frac{ (\delta^{\theta \alpha_2})^m }{ (\delta^{\alpha \alpha_1})^m } = o(\rho^m) \label{proofEndLocationOfX}
\end{align}{}
due to our choice of $\rho$ in \eqref{proofchooseRho0}\eqref{proofchooseRho}. 

Lastly, for term \eqref{proofSBAConvolutionBound}, we focus on bounding the integrand when $m \geq \Bar{m}$ where the $\bar{m}$ is the fixed integer in \eqref{proofChooseMBar}. First of all, let us focus on the domain of $x$ of the integral in \eqref{proofSBAConvolutionBound}: due to \eqref{proofChooseMbar}, for any $x < na - \delta^{m\alpha_2}$, we have 
$$na - x - \frac{\delta^m}{\sqrt{n}} \geq \delta^{m\alpha_2} - \delta^m \geq \frac{\delta^{m\alpha_2}}{2}.$$ 
Therefore, when bounding the integrand in \eqref{proofSBAConvolutionBound} (for $m \geq \bar{m}$), it suffices to consider probability of the form
$$\mathbb{P}\Big( \sum_{j = 1}^{ \ceil{\log_2(n^2)} + m }(\xi_{l_j})^+ \in [y, y+\frac{2\delta^m}{\sqrt{n}}] \Big) $$
where $y \geq \delta^{m\alpha_2}/2$. To this end, our first step is to condition on stick lengths $(l^{(i)}_1,l^{(i)}_2,\cdots)$: For an arbitrary sequence of positive real numbers $(\Tilde{t}_1,\Tilde{t}_2,\cdots,\Tilde{t}_{m +  \ceil{\log_2(n^2)}})$, denote its descending reordering as $(t_1,t_2,\cdots,t_{m +  \ceil{\log_2(n^2)}})$. By conditioning on $\{ l^{(i)}_j = \Tilde{t_j}: j = 1,2,\cdots, \ceil{\log_2(n^2)} + m \}$, we have
$$\sum_{j = 1}^{ \ceil{\log_2(n^2)} + m }(\xi_{j}^{(i)})^+ \distequal{} \sum_{j = 1}^{ \ceil{\log_2(n^2)} + m }(\widetilde{X}^\prime_{t_j})^+$$
where, for each $j$, $\widetilde{X}^\prime_{t_j}$ is an independent copy of $\widetilde{X}_{t_j}$. The idea of the proof is to consider that, for the sequence $(\widetilde{X}^\prime_{t_1},\widetilde{X}^\prime_{t_2},\cdots)$, when is the first time we get $\widetilde{X}^\prime_{t_j} > 0$. Specifically, let us set a threshold 
$$\eta = \delta^{m\alpha_3}/n^{\alpha_4},$$
and for the sequence $(t_1,t_2,\cdots)$, define the index $$J = \#\{j = 1,2,\cdots,  \ceil{\log_2(n^2)} + m:\ t_j \geq \eta\},$$
Intuitively speaking, we consider $t_j$ as a \textit{long stick} if $j \leq J$ (as we have $t_j \geq \eta$), and view it as a \textit{short stick} otherwise. By considering when is the first time we get $\widetilde{X}^\prime_{t_j} > 0$, we yield the following decomposition of events:
\begin{align}
    & \mathbb{P}\Big( \sum_{j = 1}^{ \ceil{\log_2(n^2)} + m }(\xi^{(i)}_{j})^+ \in [y, y+\frac{2\delta^m}{\sqrt{n}}] \Big) \nonumber \\
    = & \sum_{j = 1}^J \mathbb{P}\Big( \widetilde{X}^\prime_{t_l} \leq 0\ \forall l < j;\ \widetilde{X}^\prime_{t_j} > 0; \ \sum_{l = j}^{ \ceil{\log_2(n^2)} + m }(\widetilde{X}^\prime_{t_l} )^+ \in [y, y+\frac{2\delta^m}{\sqrt{n}}]   \Big) \label{proofSBAconvolutionLongStick} \\
    + & \mathbb{P}\Big( \widetilde{X}^\prime_{t_l} \leq 0\ \forall l\leq J; \ \sum_{l = J+1}^{ \ceil{\log_2(n^2)} + m }(\widetilde{X}^\prime_{t_l})^+ \in [y, y+\frac{2\delta^m}{\sqrt{n}}]  \Big). \label{proofSBAconvolutionShortStick}
\end{align}{}
Again, we tackle the two terms respectively. For term \eqref{proofSBAconvolutionLongStick}, consider a fixed $j = 1,2,\cdots,J$ and note that
\begin{align}
     & \mathbb{P}\Big( \widetilde{X}^\prime_{t_l} \leq 0\ \forall l < j;\ \widetilde{X}^\prime_{t_j} > 0; \ \sum_{l = j}^{ \ceil{\log_2(n^2)} + m }(\widetilde{X}^\prime_{t_l} )^+ \in [y, y+\frac{2\delta^m}{\sqrt{n}}]   \Big) \nonumber \\
     \leq & \mathbb{P}\Big( \widetilde{X}^\prime_{t_j} > 0; \ \sum_{l = j}^{ \ceil{\log_2(n^2)} + m }(\widetilde{X}^\prime_{t_l} )^+ \in [y, y+\frac{2\delta^m}{\sqrt{n}}]   \ \Big) \nonumber \\
     = & \int_{\mathbb{R}}\mathbb{P}\Big(  \widetilde{X}^\prime_{t_j} \in \big(0 \vee (y - x), 0 \vee (y-x + \frac{2\delta^m}{\sqrt{n}}) \big) \Big) \mathbb{P}\Big( \sum_{l = j+1}^{ \ceil{\log_2(n^2)} + m}( \widetilde{X}^\prime_{t_l})^+ = dx \Big) \nonumber \\
     \leq & 2^\theta C\frac{n^{\alpha \alpha_4}(\delta^\theta)^m }{n^{\theta/2}(\delta^{\alpha\alpha_3})^m} \ \ \ \text{due to Assumption \ref{assumption_Xdist} and $t_j \geq \eta$.} \nonumber
\end{align}{}
Since $J \leq m + \ceil{\log_2(n^2)}$, we have
\begin{align}
    & \eqref{proofSBAconvolutionLongStick} \leq 2^\theta C \frac{(m + \ceil{\log_2(n^2)})n^{\alpha \alpha_4} }{n^{\theta/2}} \cdot (\delta^{\theta - \alpha \alpha_3})^m \nonumber \\
    = & 2^\theta C \Big( \frac{m(\delta^{\theta - \alpha \alpha_3})^m }{n^{\frac{\theta}{2} - \alpha\alpha_4}} + \frac{\ceil{\log_2(n^2)}}{ n^{\frac{\theta}{2} - \alpha\alpha_4}  }m(\delta^{\theta - \alpha \alpha_3})^m \Big) \nonumber \\
    = & o(\rho^m) \label{proofEndSBAconvolutionLongStick}
\end{align}{}
due to our choice of $\rho$ in \eqref{proofchooseRho0}\eqref{proofchooseRho} and choice of $\alpha_4$ in \eqref{proofChooseAlpha34}.

On the other hand, for term \eqref{proofSBAconvolutionShortStick}, we have (recall that $y \geq \delta^{m\alpha_2}/2$)
\begin{align}
    & \mathbb{P}\Big( \widetilde{X}^\prime_{t_l} \leq 0\ \forall l\leq J; \ \sum_{l = J+1}^{ \ceil{\log_2(n^2)} + m }(\widetilde{X}^\prime_{t_l})^+ \in [y, y+\frac{2\delta^m}{\sqrt{n}}]  \Big) \nonumber \\
    \leq & \mathbb{P}\Big( \sum_{l = J+1}^{ \ceil{\log_2(n^2)} + m }(\widetilde{X}^\prime_{t_l})^+ \in [y, y+\frac{2\delta^m}{\sqrt{n}}]   \Big) \nonumber \\
    \leq &  \mathbb{P}\Big( \sum_{l = J+1}^{ \ceil{\log_2(n^2)} + m }(\widetilde{X}^\prime_{t_l})^+ \geq y  \Big) \nonumber \\
    \leq & \sum_{l = J+1}^{ \ceil{\log_2(n^2)} + m }\mathbb{P}\Big( \widetilde{X}^\prime_{t_l} \geq \frac{y}{L} \Big) \ \ \ \text{(Let $L  =  \ceil{\log_2(n^2)} + md - J$)} \nonumber \\
    \leq & \sum_{l = J+1}^{ \ceil{\log_2(n^2)} + m }C_X\frac{(t_l + \sqrt{t_l})L}{y} \ \ \ \text{due to Markov's inequality, $t_l \leq \eta$, and Result \ref{resultSBA}; $C_X$ is the constant in \eqref{proofFirstUseSBAMomentBound}} \nonumber \\
    \leq & C_X\frac{( \ceil{\log_2(n^2)} + m )^2( \sqrt{ \delta^{\alpha_3}/n^{\alpha_4}} +  \delta^{\alpha_3}/n^{\alpha_4} )}{\delta^{m\alpha_2}} \nonumber \\
    \leq & 2C_X\frac{ ( \ceil{\log_2(n^2)} + m )^2 \delta^{\alpha_3/2} }{ n^{\alpha_4/2}\delta^{m\alpha_2} } \nonumber \\
    \leq & 4C_X\Big( \frac{ \ceil{\log_2(n^2)}^2 }{n^{\alpha_4/2}}(\delta^{-\alpha_2 + \frac{\alpha_3}{2} })^m + m^2(\delta^{-\alpha_2 + \frac{\alpha_3}{2}})^m \Big)\ \ \ \text{due to $(u+v)^2\leq 2u^2 + 2v^2$} \nonumber \\
    = & o(\rho^m)\label{proofEndSBAconvolutionShortStick}
\end{align}
due to our choice of $\rho$ in \eqref{proofchooseRho0}\eqref{proofchooseRho}.

At this point, we collect previous results and return to the question of bounding the term \eqref{proofIneqFixedI}. The discussion above shows that for $i \geq 1$, the error term is $o(\rho^m)$. To see why the same bound holds for the case $i=0$, we only need to notice that, in this case, term \eqref{proofLocationOfX} can be removed from the inequality since $u_0 = 0$ and $X_0 = 0$ almost surely, and bounds on the other terms would still apply. Therefore, in \eqref{proofIneqUnbiasEstL2withKjumps}, we have
\begin{align}
    \mathbb{P}( Y_{n, m }(\zeta_k) \neq Y^*_n(\zeta_k)  ) = k\cdot o(\rho^m), \nonumber
\end{align}{}
which implies the existence of some constant $\Tilde{C} < \infty$ such that
\begin{align}
    \mathbb{E}[ Z^2_{n}(J_n)  |\ J_n = \zeta_k ] \leq \Tilde{C}k \label{proofEndZK}
\end{align}

To complete the proof, recall that $\lambda_n = n\nu[n\gamma,\infty)$, and the fact that the number of jumps of $J_n$ on $[0,n]$ follows a Poisson distribution with rate $\lambda_n$. Now we have
\begin{align}
    \mathbb{E}Z^2_{n,1} & \leq \sum_{k \geq l^*} \tilde{C}\cdot k \cdot e^{-\lambda_n}\frac{\lambda_n^k}{k!} \nonumber \\
    & \leq \tilde{C} \lambda_n^{l^*} \sum_{k \geq l^*}e^{-\lambda_n}\frac{\lambda_n^{k - l^*} }{(k-1)!} \nonumber \\
    & \leq \tilde{C} \lambda_n^{l^*} \sum_{k \geq l^*} e^{-\lambda_n}\frac{\lambda_n^{k - l^*} }{(k-l^*)!} \ \ \ \text{due to $l^* \geq 2$} \nonumber \\
    & = \tilde{C}\lambda^{l^*}_n = \tilde{C}\cdot( n\nu[n\gamma,\infty) )^{l^*}. \label{proof_BoundSquaredZ1}
\end{align}{}
Lastly, note that 
$$\Big( n\nu[n\gamma,\infty) \Big)^{l^*} = \mathcal{O}\Big( ( n\nu[n,\infty) )^{l^*} \Big) $$
due to the fact that $f^+(x) = \nu[x,\infty)$ is regularly varying with index $-\alpha_+$ where $\alpha_+ > 1$; on the other hand, note that
$$ \liminf_{n \rightarrow \infty} \frac{\mathbb{P}( \bar{X}_n \in A)}{( n\nu[n,\infty) )^{l^*}} \geq C^{l^*}_{\alpha_+}(A^o) > 0 $$
thanks to Assumption \ref{assumption_setA} and Result \ref{resultLevyLD}. Therefore, we now have
$$\mathbb{E}Z^2_{n,1} = \mathcal{O}(\mathbb{P}(A_n)),$$
which establishes \eqref{proofGoal1}, and concludes the first half of the proof. Note that: regardless of the value of $\gamma > 0$, claim \eqref{proofGoal1} always holds as long as the other parameters are properly chosen as specified at the beginning of the proof.

Moving on, we show that $\mathbb{E}Z^2_{n,2} = \mathcal{O}(\mathbb{P}^2(A_n))$. By definition of $Y_{n,m}$ and $Y^*_n$, we have $Y^*_n \geq Y_{n,m}$; then since they are all indicator functions, for $\{Y_{n,m} \neq Y^*_n\}$ to occur we need to at least have $Y^* = 1$, which implies $\Bar{X}_n \in A$. Therefore, using Result \ref{resultDebias}, when conditioning on the event $\{J_n = \zeta_k\}$ we have (for any $k = 0,1,\cdots,l^* - 1$)
\begin{align*}
    \mathbb{E}[Z^2_{n,2}|J_n = \zeta_k] & \leq \sum_{m \geq 1}\frac{ \mathbb{P}\Big( Y_{n,m-1}(\zeta_k) \neq Y^*_n(\zeta_k);\ \widetilde{X}+\zeta_k \in A_n\Big)   }{ \mathbb{P}(\tau \geq m)  } \\
    & \leq \sqrt{ \mathbb{P}\Big(\widetilde{X}+\zeta_k \in A_n \Big)  }\sum_{m \geq 1}\frac{ \sqrt{\mathbb{P}(Y^*_{n}(\zeta_k) \neq Y_{n,m}(\zeta_k))} }{ \mathbb{P}(\tau \geq m)} \\
    \leq & C_2 \sqrt{ \mathbb{P}\Big(\widetilde{X}+\zeta_k \in A_n\Big)  }
\end{align*}
due to Hölder's inequality, \eqref{proofEndZK}, and our choice of $\sqrt{\rho_0} < \rho$. Here $C_2$ is some finite constant. Therefore, unconditionally, we have (due to Jensen's inequality)
\begin{align*}
    \mathbb{E}Z^2_{n,2} & \leq C_2 \sqrt{ \mathbb{P}( \Bar{X}_n \in A,\ \Bar{X}_n \text{ has no more than $l^* - 1$ jumps larger than }\gamma ) } \\
    & = C_2 \sqrt{ \mathbb{P}( \Bar{X}_n \in A\cap (B^\gamma)^c) } \\
    & = \mathcal{O}( \mathbb{P}^2(A_n) )
\end{align*}
due to Result \ref{resultLevyLD} and our choice of $\gamma$ in \eqref{proofChooseGamma}. To see why, note that due to our choice of $\gamma$, if a step function $\zeta \in A\cap (B^\gamma)^c)$, then $\zeta$ must have at least $4l^*$ jumps, which implies
$$A\cap (B^\gamma)^c \cap \mathbb{D}_{<4l^*} = \emptyset.$$
Then from Result \ref{resultLevyLD}, one can see that (for some constant $C_3 < \infty,C_4>0$) such that for any $n$:
$$\sqrt{ \mathbb{P}( \Bar{X}_n \in A\cap (B^\gamma)^c) } \leq C_3 \big( n\nu[n,\infty) \big)^{2l^*}, \ \ \ \mathbb{P}^2(A_n) \geq C_4 \big( n\nu[n,\infty) \big)^{2l^*}.$$
This concludes the proof.
\end{proof}

\section{Proof of Theorem \ref{thmDistX}} \label{appendixProof2}

To prove Theorem \ref{thmDistX}, we first show the following Lemma.
\begin{lemma}
Given $\alpha \in (0,2),\gamma_0 >0,\epsilon \in (0, (2 - \alpha)/2)$, and a measure $\nu$ concentrated on $(0,\infty)$ such that $f(x) = \nu[x,\infty)$ is regularly varying at $0$ with index $-(\alpha+2\epsilon)$, there exists a constant $C < \infty$ such that for the L\'evy process $\{X_t:t\geq 0\}$ with generating triplet $(0,0,\nu|_{(-\gamma_0,\gamma_0)})$, we have
\begin{align*}
\norm{f_{X_t}}_\infty \leq \frac{C}{t^{1/\alpha} \wedge 1}\ \ \forall t > 0.
\end{align*}
where $\norm{\cdot}_\infty$ is the $L_\infty$ norm, and $f_{X_t}$ is the density function of the distribution of $X_t$.
\end{lemma}

\begin{proof}

We start by fixing some notations and parameters. Let
\begin{align}
    C_0 = \int_{0}^\infty (1 - \cos{y})\frac{dy}{y^{1 + \alpha}}.
\end{align}
Apparently, we have $C_0 \in (0,\infty)$. Besides, choose positive real numbers $\theta,\delta$ such that:
\begin{align}
    \frac{\theta^{2 - \alpha}}{2(2 - \alpha)}\leq \frac{C_0}{8}; \label{chooseTheta_Lcont} \\
    \frac{\delta}{\theta^\alpha}\leq \frac{C_0}{8}. \label{chooseDelta_Lcont}
\end{align}

The idea is to appeal to the inversion formula and show that $X_t$ has a bounded, continuous density, then we show how the $L_\infty$ norm of density of $X_t$ scales with $t$. To this end, we first apply Lévy-Khintchine formula: for any $t>0$, the characteristic function of $X_t$ is
\begin{align*}
    \varphi_{t}(z) = \exp\Big( t\int_{(0,\gamma_0)}\big(\exp(izx) -1 - izx\mathbbm{1}_{(0,1]}(x)\big)\nu(dx)\Big)\ \ \forall z\in\mathbb{R}.
\end{align*}

By considering the complex conjugate of $\int_{(0,\gamma_0)}\big(\exp(izx) -1 - izx\mathbbm{1}_{(0,1]}(x)\big)\nu(dx)$, we see that
\begin{align}
    |\varphi_t(z)| = \exp\Big(-t\int_{(0,\gamma_0)}\big(1 - \cos(zx)\big)\nu(dx) \Big). \label{normOfCF_lcont}
\end{align}

For any $M > 0$, we have (for any $z \neq 0$)
\begin{align*}
    \frac{\int_{x \geq \frac{M}{|z|}}\big( 1 - \cos(zx)  \big)\frac{dx}{x^{1 + \alpha}}  }{|z|^\alpha} & = \frac{\int_{x \geq \frac{M}{|z|}}\big( 1 - \cos(|z|x)  \big)\frac{dx}{x^{1 + \alpha}}  }{|z|^\alpha} \\
    &= \int_{M}^\infty \big(  1 -\cos{y} \big)\frac{dy}{y^{1+\alpha}}\ \ \ \text{by letting $y = |z|x$.}
\end{align*}
Therefore, we can fix some $M > \theta$ such that (for any $z \neq 0$)
\begin{align}
    \frac{1}{|z|^\alpha}\int_{x \geq M/|z|}\big( 1 - \cos(zx)  \big)\frac{dx}{x^{1 + \alpha}} \leq C_0/4. \label{endNonRVterm_lCont}
\end{align}

Moving on, we fix some $c > 0$ and consider the difference between $\int_{(0,\gamma_0)}\big( 1 - \cos(zx) \big)\nu(dx)$ and $\int_0^{M/z}\big(1 - \cos(zx)\big)\frac{c dx}{x^{1 + \alpha}}$. First, for any $z$ such that $|z|>M/\gamma_0$, we have
\begin{align}
    & \frac{1}{|z|^\alpha}\Big[ \int_{(0,\gamma_0)}\Big( 1 - \cos(zx) \Big)\nu(dx) - \int_0^{M/|z|}\Big( 1 - \cos(zx)\Big)c\frac{dx}{x^{1 + \alpha}} \Big] \nonumber \\
    \geq & \frac{1}{|z|^\alpha}\Big[ \int_{(0,M/|z|)}\Big( 1 - \cos(zx) \Big)\nu(dx) - \int_0^{M/|z|}\Big( 1 - \cos(zx)\Big)c\frac{dx}{x^{1 + \alpha}} \Big]\nonumber \\
    \geq & -\frac{1}{|z|^\alpha}\int_{\theta/|z|}^{M/|z|}\Big( 1 - \cos(zx)\Big)c\frac{dx}{x^{1 + \alpha}} \label{term1_RV_LCont} \\
    + & \frac{1}{|z|^\alpha}\Big[\int_{[\theta/|z|,M/|z|)}\Big( 1 - \cos(zx) \Big)\nu(dx) - \int_{\theta/|z|}^{M/|z|}\Big( 1 - \cos(zx)\Big)c\frac{dx}{x^{1 + \alpha}}   \Big]. \label{term2_RV_LCont}
\end{align}
For \eqref{term1_RV_LCont}, by letting $y = |z|x$ we have
\begin{align}
   & \frac{1}{|z|^\alpha}\int_{0}^{\theta/|z|}\big( 1 - \cos(|z|x) \big)c\frac{dx}{x^{1 + \alpha}} \leq \frac{c}{|z|^\alpha}\int_{0}^{\theta/|z|}\frac{z^2x^2}{2}\frac{dx}{x^{1+\alpha}}\nonumber  \\
   =& \frac{c}{2}\int_0^\theta y^{1 - \alpha} dy = \frac{c}{2}\cdot \frac{ \theta^{2 - \alpha} }{2 - \alpha}\nonumber \\
   & \leq c\cdot\frac{C_0}{8}\ \ \ \text{due to \eqref{chooseTheta_Lcont}}. \label{endTerm1RV_lCont}
\end{align}
For \eqref{term2_RV_LCont}, let us focus on the function on $\mathbb{R}$:
$$h(z) = 1 - \cos{z}.$$
Since $h(z)$ is uniformly continuous on $[\theta,M]$, we can find a sequence of real numbers $\{x_k \}_{k = 0}^{N}$ with $N \in \mathbb{N}, t_0 > 1$ such that
\begin{align}
   & x_0 = M, x_N = \theta;\nonumber \\
    &\frac{x_{j-1}}{x_j} = t_0\ \ \forall j = 1,2,\cdots,N;\nonumber \\
    & |h(x) - h(y)| < \delta\ \ \ \forall j = 1,2,\cdots,N,\ \forall x,y \in [x_j,x_{j-1}]. \label{uContOfG_lCont}
\end{align}
In other words, we use a geometrically decreasing sequence of points $\{x_0,x_1,\cdots,x_{N}\}$ to partition $[\theta,M]$ into $N+1$ intervals, on any of which the value of $h(z) = 1 - \cos{z}$ would not vary beyond the $\delta$ chosen in \eqref{chooseDelta_Lcont}.

At this step, fix some $\Delta > 0$ such that
\begin{align}
    (1 - \Delta)t_0^{\alpha+\epsilon} > 1. \label{chooseCapDelta_LCont}
\end{align}
Recall that for $g(y) = \nu[1/y,\infty)$, we have $g\in RV_{\alpha + 2\epsilon}$. Therefore, by Potter bounds (see Proposition 2.6 in \cite{resnick2007heavy}), we know the existence of some $Y_1 > 0$ such that for any $ t \geq 1$, we have
\begin{align}
    \frac{g(ty)}{g(y)} \geq (1 - \Delta)t^{\alpha + \epsilon} \ \ \forall y \geq Y_1. \label{potterBound_lCont}
\end{align}
On the other hand, define
$$\Tilde{g}(y) = cy^\alpha,$$
and we know that $\Tilde{g}(y) = \nu_c(1/y,\infty)$ where $\nu_c(dx) = c\mathbbm{1}_{(0,
\infty)}(x)\frac{dx}{x^{1+\alpha}}$.
The fact that $g \in RV_{\alpha + 2\epsilon}$ implies the existence of some $Y_2 > 0$ such that
\begin{align}
    g(y)\geq \frac{t_0^\alpha - 1}{ (1 - \Delta)t_0^{\alpha + \epsilon} - 1 }\cdot  \Tilde{g}(y)\ \ \ \forall y \geq Y_2. \label{gBound_lCont}
\end{align}
Let $\Tilde{M} = \max\{ M/\gamma_0, MY_1, MY_2 \}$. For $z \in \mathbb{R}$ with $|z| \geq \Tilde{M}$ and any $j = 1,2,\cdots,N$, the mass of $\nu$ on $[x_j,x_{j-1})$ is
\begin{align*}
    \nu[x_j/|z|, x_{j-1}/|z|) & = g(|z|/x_{j}) - g(|z|/x_{j-1}) \\
    & = g(t_0|z|/x_{j-1}) - g(|z|/x_{j-1}) \\
    & \geq g(|z|/x_{j-1})\cdot\Big( (1 - \Delta)t_0^{\alpha+\epsilon} - 1 \Big)\ \ \ \text{due to \eqref{potterBound_lCont}} \\
    & \geq \Tilde{g}(|z|/x_{j - 1})\cdot( t_0^\alpha - 1 )\ \ \ \text{due to \eqref{gBound_lCont}}
\end{align*}
while the mass of $\nu_c$ on $[x_j,x_{j-1})$ is
\begin{align*}
    \nu_c[x_j/|z|,x_{j - 1}/|z|) & = \Tilde{g}(|z|/x_j) - \Tilde{g}(|z|/x_{j-1}) \\
    & = (t_0^\alpha - 1)\Tilde{g}(|z|/x_{j-1}).
\end{align*}{}
Therefore, for any specific $z \in \mathbb{R}$ with $|z|\geq \Tilde{M}$, we can construct a measure $\nu^{(z)}$ such that
\begin{itemize}
    \item $\nu - \nu^{(z)}$ is a positive measure (in the sense of signed measure);
    \item for each $j = 1,2,\cdots,N$, for $\nu_c(dx) = c\mathbbm{1}_{(0,
\infty)}(x)\frac{dx}{x^{1+\alpha}}$, we have 
\begin{align}
    \nu^{(z)}[x_j,x_{j-1}) = \nu_c[x_j,x_{j-1}). \label{nu_z_construction_LCont}
\end{align}{}
\end{itemize}{}

Now in \eqref{term2_RV_LCont}, we have
\begin{align}
    & \frac{1}{|z|^\alpha}\Big[\int_{[\theta/|z|,M/|z|)}\Big( 1 - \cos(zx) \Big)\nu(dx) - \int_{\theta/|z|}^{M/|z|}\Big( 1 - \cos(zx)\Big)c\frac{dx}{x^{1 + \alpha}}   \Big] \nonumber \\
    =& \frac{1}{|z|^\alpha}\sum_{j = 1}^N\Big[  \int_{[x_j/|z|,x_{j-1}/|z|)}\Big( 1 - \cos(zx) \Big)\nu(dx) - \int_{x_j/|z|}^{x_{j-1}/|z|}\Big( 1 - \cos(zx)\Big)c\frac{dx}{x^{1 + \alpha}}  \Big]\nonumber \\
    \geq & \frac{1}{|z|^\alpha}\sum_{j = 1}^N\Big[  \int_{[x_j/|z|,x_{j-1}/|z|)}\Big( 1 - \cos(zx) \Big)\nu^{(z)}(dx) - \int_{x_j/|z|}^{x_{j-1}/|z|}\Big( 1 - \cos(zx)\Big)c\frac{dx}{x^{1 + \alpha}}  \Big]\ \ \text{due to $\nu - \nu^{(z)} \geq 0$}\nonumber \\
    \geq &  -\frac{\delta}{|z|^\alpha}\int_{\theta/|z|}^{M/|z|}c\frac{dx}{x^{1+\alpha}}\ \ \ \text{due to \eqref{uContOfG_lCont},\eqref{nu_z_construction_LCont}} \nonumber \\
    \geq & -\frac{\delta}{|z|^\alpha}\int_{\theta/|z|}^{\infty}c\frac{dx}{x^{1+\alpha}} = -c\delta/\theta^\alpha \nonumber \\
    \geq & -c\cdot\frac{C_0}{8}\ \ \ \text{due to \eqref{chooseDelta_Lcont}.} \label{endTerm2RV_lCont}
\end{align}{}

Plugging \eqref{endTerm1RV_lCont} and \eqref{endTerm2RV_lCont} back into \eqref{term1_RV_LCont} and \eqref{term2_RV_LCont} and also using \eqref{endNonRVterm_lCont}, we have that: for any $z\in\mathbb{R}$ with $|z|\geq\Tilde{M}$,
\begin{align}
    & \frac{1}{|z|^\alpha}\Big[   \int_{(0,\gamma_0)}\Big( 1 - \cos(zx) \Big)\nu(dx) - \int_0^{\infty}\Big( 1 - \cos(zx)\Big)c\frac{dx}{x^{1 + \alpha}}  \Big] \nonumber \\
    \geq & -cC_0(\frac{1}{4} + \frac{1}{8} + \frac{1}{8})\nonumber \\
    \geq & -\frac{C_0}{2}c;
\end{align}{}
on the other hand, by letting $y = |z|x$ again we have
\begin{align}
    & \int_0^{\infty}\Big( 1 - \cos(zx)\Big)c\frac{dx}{x^{1 + \alpha}} = c|z|^\alpha\int_{0}^\infty \big( 1 - \cos{y} \big)\frac{dy}{y^{1 + \alpha}}    \nonumber \\
    & = cC_0.
\end{align}{}

In summary, we have shown that: for any $t > 0$ and any $z$ with $|z| \geq \Tilde{M}$,
\begin{align*}
    |\varphi_t(z)| & = \exp\Big(-t\int_{(0,\gamma_0)}\big(1 - \cos(zx)\big)\nu(dx) \Big) \\
    & \leq \exp\big( -t\cdot\frac{cC_0}{2}|z|^\alpha \big).
\end{align*}{}

Lastly, by applying inversion formula, we know that: for any $t > 0$,
\begin{align*}
    \norm{f_{X_t}}_\infty & \leq \frac{1}{2\pi}\int |\varphi_t(z)|dz \\
    & \leq \frac{1}{2\pi}\Big( 2\Tilde{M} + \int_{|z| \geq \Tilde{M}}\exp\big( -t\cdot\frac{cC_0}{2}|z|^\alpha \big)dz \Big) \\
    & \leq \frac{1}{2\pi}\Big( 2\Tilde{M} + \int \exp\big( -t\cdot\frac{cC_0}{2}|z|^\alpha \big)dz \Big)\\
    & = \frac{1}{2\pi}\Big( 2\Tilde{M} + \frac{1}{t^{1/\alpha}}\int \exp(-\frac{cC_0}{2}|x|^\alpha)dx \Big)\ \ \ \text{by letting $x = zt^{1/\alpha}$}\\
    & = \frac{\Tilde{M}}{\pi} + \frac{C_1}{t^{1/\alpha}}
\end{align*}{}
where $C_1 = \int \exp(-\frac{cC_0}{2}|x|^\alpha)dx < \infty.$ We conclude the proof by stating that the constant $C$ in the claim of the Lemma can be set as
$$C = \frac{\Tilde{M}}{\pi} + C_1.\ \ \ $$
\end{proof}
Now we can see that Theorem \ref{thmDistX} follows immediately from the Lemma above. Indeed, if for a L\'evy process $X$ with generating triplet $(0,0,\nu)$, we know that, for some $\gamma_0 > 0$, we have $ (\nu - \mu)|_{(-\gamma_0,\gamma_0)} \geq 0$ where $\mu$ is a Borel measure on $\mathbb{R}$ and $\mu\big( (-\infty,-x]\cup[x,\infty) \big)$ is regularly varying at 0 with index $-(\alpha + \epsilon)$ for some $\alpha \in (0,2),\epsilon \in (0,2 - \alpha)$, then consider the following decomposition: $X_t = Y_t + Z_t$ where $\{Y_t: t \geq 0\}$ and $\{ Z_t: t \geq 0\}$ are two independent L\'evy processes with generating triplets $(0,0,\mu|_{(-\gamma_0,\gamma_0)})$ and $(0,0,\nu - (\mu)|_{(-\gamma_0,\gamma_0)})$ respectively. The discussion above implies the existence of $C>0$ such that
$$\norm{ f_{Y_t} }_\infty \leq \frac{C}{t^{1/\alpha}\wedge 1}\ \ \forall t > 0.$$
Now we have: for any $t>0$, $x \in \mathbb{R}$ and $\delta > 0$, 
\begin{align}
    \mathbb{P}(X_t \in [x,x+\delta]) & = \int_{\mathbb{R}}\mathbb{P}(Y_t \in [x-z,x-z+\delta])\mathbb{P}(Z_t = dz)\nonumber \\
    & \leq \frac{C}{t^{1/\alpha}\wedge 1}\delta \nonumber
\end{align}
which gives us exactly the bound in Theorem \ref{thmDistX}.

\end{document}